\documentclass[a4paper]{amsart}

\usepackage{color,graphicx,enumerate,wrapfig,amssymb,mathtools}
\usepackage{epsfig}

\usepackage{eucal}

\newcommand{\R}{{\mathbb R}}

\newcommand{\N}{{\mathbb N}}

\newcommand{\Hf}{{\mathbb{H}}}

\newcommand{\cB}{{\mathcal B}}

\newcommand{\cR}{{\mathcal R}}

\newcommand{\uu}{\mathbf{u}}
\newcommand{\vv}{\mathbf{v}}
\newcommand{\ww}{\mathbf{w}}
\newcommand{\cc}{\mathbf{c}}
\newcommand{\h}{\mathbf{h}}
\newcommand{\g}{\mathbf{ g}}
\newcommand{\qq}{\mathbf{q}}
\newcommand{\pp}{\mathbf{p}}
\newcommand{\wuuj}{\tilde{\mathbf{u}}_j}
\newcommand{\wuuz}{\tilde{\mathbf{u}}_0}
\newcommand{\e}{\varepsilon}
\newcommand{\al}{\alpha}
\newcommand{\be}{\beta}
\newcommand{\de}{\delta}
\newcommand{\ka}{\kappa}
\newcommand{\la}{\lambda}
\newcommand{\si}{\sigma}
\newcommand{\vp}{\varphi}

\newcommand{\dist}{\operatorname{dist}}

\newcommand{\loc}{\operatorname{loc}}

\newcommand{\D}{\nabla}
\newcommand{\p}{\partial}

\newcommand{\mean}[1]{\langle{#1}\rangle}

\usepackage{amsmath}
\usepackage{amsfonts}
\usepackage{caption}
\usepackage{subcaption}
\usepackage{esint}

\newtheorem{theorem}{Theorem}
\theoremstyle{plain}
\newtheorem{corollary}{Corollary}
\newtheorem{definition}{Definition}
\newtheorem{example}{Example}
\newtheorem{lemma}{Lemma}
\newtheorem{remark}{Remark}
\newtheorem{proposition}{Proposition}

\numberwithin{equation}{section}

\makeatletter
\@namedef{subjclassname@2020}{%
  \textup{2020} Mathematics Subject Classification}
\makeatother

\begin{document}

\author{Daniela De Silva}
\address[Daniela De Silva]
{Department of Mathematics \newline 
\indent  Barnard College, Columbia University \newline 
\indent New York, NY,
10027} 
\email[Daniela De Silva]{desilva@math.columbia.edu}

\author{Seongmin Jeon}
\address[Seongmin Jeon]
{Department of Mathematics \newline 
\indent KTH Royal Institute of Technology \newline 
\indent 100 44 Stockholm, Sweden} 
\email[Seongmin Jeon]{seongmin@kth.se}

\author{Henrik Shahgholian}
\address[Henrik Shahgholian]
{Department of Mathematics \newline 
\indent KTH Royal Institute of Technology \newline 
\indent 100 44 Stockholm, Sweden} 
\email[Henrik Shahgholian]{henriksh@kth.se}

\title[Almost minimizers for a sublinear system with free boundary]
{Almost minimizers for a sublinear system with free boundary}

\date{\today}
\thanks{ DD is supported by a NSF grant (RTG 1937254).
 HS is supported by Swedish Research Council.}

\begin{abstract}
    We study vector-valued almost minimizers of the energy functional $$
    \int_D\left(|\D\uu|^2+\frac2{1+q}\left(\la_+(x)|\uu^+|^{q+1}+\la_-(x)|\uu^-|^{q+1}\right)\right)dx,\quad0<q<1.
    $$
    For H\"older continuous coefficients $\la_\pm(x)>0$, we take the epiperimetric inequality approach and prove the regularity for both almost minimizers and the set of ``regular" free boundary points.
\end{abstract}

\maketitle 

\tableofcontents

\section{Introduction}

\subsection{A sublinear elliptic system}\label{subsec:system}
Let $D$ be a domain in $\R^n$, $n\ge2$, and $\g:\p D\to\R^m$, $m\ge1$, be a given function (the boundary value). We also let $\la_\pm(x)$ be $\al$-H\"older continuous functions satisfying $0<\la_0\le\la_\pm(x)\le\la_1<\infty$ for some positive constants $\la_0$, $\la_1$. Then, for $0<q<1$ and $F(x,\vv)=\frac1{1+q}\left(\la_+(x)|\vv^+|^{q+1}+\la_-(x)|\vv^-|^{q+1}\right)
$, consider the minimizer $\uu$ of the energy functional \begin{align}
    \label{eq:energy-ftnal}
    \int_D\left(|\D\vv|^2+2F(x,\vv)\right)dx,
\end{align}
among all functions $\vv\in W^{1,2}(D;\R^m)$ with $\vv=\g$ on $\p D$. It is well-known that there exists a unique minimizer $\uu$ and it solves a sublinear system $$
\Delta\uu=\la_+(x)|\uu^+|^{q-1}\uu^+-\la_-(x)|\uu^-|^{q-1}\uu^-.
$$
The regularity of both the solution $\uu$ and its free boundary $\Gamma(\uu):=\p\{x\,:\,|\uu(x)|>0\}$ was studied in \cite{FotShaWei21} or in the scalar case (when $m=1$) in \cite{FotSha17}.


\subsection{Almost minimizers}
In this paper we consider almost minimizers of the functional \eqref{eq:energy-ftnal}.\\
\indent To introduce the definition of almost minimizers, we let $\omega:(0,r_0)\longmapsto[0,\infty)$, $r_0>0$, be a \emph{gauge} function, which is a nondecreasing function with $\omega(0+)=0$.
\begin{definition}
[Almost minimizers]\label{def:alm-min}
Let $0<r_0<1$ be a constant and $\omega(r)$ be a gauge function. We say that a function $\uu\in W^{1,2}(B_1;\R^m)$ is an almost minimizer for the functional $\int\left(|\D\uu|^2+2F(x,\uu)\right)\,dx$ in a domain $D$, with gauge function $\omega(r)$, if for any ball $B_r(x_0)\Subset D$ with $0<r<r_0$, we have \begin{align}
    \label{eq:alm-min}
    \int_{B_r(x_0)}\left(|\D\uu|^2+2F(x,\uu)\right)dx\le(1+\omega(r))\int_{B_r(x_0)}\left(|\D\vv|^2+2F(x,\vv)\right)dx,
\end{align}
for any competitor function $\vv\in \uu+W^{1,2}_0(B_r(x_0);\R^m)$.
\end{definition}
In fact, we can observe that for $x$, $x_0\in D$, \begin{align}
    \label{eq:lambda-ratio}
1-C|x-x_0|^\al\le\frac{\la_\pm(x)}{\la_\pm(x_0)}\le 1+C|x-x_0|^\al,
\end{align}
with a constant $C$ depending only on $\la_0$ and $\|\la_\pm\|_{C^{0,\al}(D)}$. Using this, we can rewrite \eqref{eq:alm-min} in the form with frozen coefficients \begin{align}
    \label{eq:alm-min-frozen}
\int_{B_r(x_0)}\left(|\D\uu|^2+2F(x_0,\uu)\right)dx\le (1+\widetilde{\omega}(r))\int_{B_r(x_0)}\left(|\D\vv|^2+2F(x_0,\vv)\right)dx,
\end{align}
where \begin{align*}
    F(x_0,\uu)&=\frac1{q+1}\left(\la_+(x_0)|\uu^+|^{q+1}+\la_-(x_0)|\uu^-|^{q+1}\right),\\
    \widetilde{\omega}(r)&=C(\omega(r)+r^\al).
\end{align*}
This implies that almost minimizers of \eqref{eq:energy-ftnal} with H\"older coefficients $\la_\pm$ are almost minimizers with frozen coefficients \eqref{eq:alm-min-frozen}.\\
\indent An example of an almost minimizer can be found in Appendix~\ref{sec:ex}. Almost minimizers for the case $q=0$ and $\la_\pm=1$ were studied by the authors in \cite{DeSJeoSha21}, where the regularity of both the almost minimizers and the regular part of the free boundary has been proved.\\
\indent In this paper we aim to extend the results in \cite{FotShaWei21} from solutions to almost minimizers and those in \cite{DeSJeoSha21} from the case $q=0$ to $0<q<1$.


\subsection{Main results}
For the sake of simplicity, we assume that the gauge function $\omega(r)=r^\al$ for $0<\al<2$, $D$ is the unit ball $B_1$, and the constant $r_0=1$ in Definition~\ref{def:alm-min}.\\
\indent In addition, to simplify tracking all constants, we take $M>2$ such that \begin{align}\label{eq:M}
\|\la_\pm\|_{C^{0,\al}(B_1)}\le M,\qquad \frac1{\la_0},\,\la_1\le M,\qquad \omega(r),\,\widetilde{\omega}(r)\le Mr^\al.
\end{align}
Now, we state our main results.

\begin{theorem}[Regularity of almost minimizers]\label{thm:grad-u-holder}
Let $\uu\in W^{1,2}(B_1;\R^m)$ be an almost minimizer in $B_1$. Then $\uu\in C^{1,\al/2}(B_1)$. Moreover, for any $K\Subset B_1$, \begin{align}\label{eq:grad-u-holder}
\|\uu\|_{C^{1,\al/2}(K;\R^m)}\le C(n,\al,M,K)(E(\uu,1)^{1/2}+1),
\end{align}
where $E(\uu,1)=\int_{B_1}\left(|\D \uu|^2+|\uu|^{q+1}\right)$.
\end{theorem}
This regularity result is rather immediate in the case of minimizers (or solutions), since their $W^{2,p}$-regularity for any $p<\infty$ simply follows from the elliptic theory with a bootstrapping. This is inapplicable to almost minimizers, as they do not satisfy a partial differential equations. Instead, we follow the approach in \cite{DeSJeoSha21} by first deriving growth estimates for almost minimizers and then using Morrey and Campanato space embedding theorems.

To investigate the free boundary, for $\ka:=\frac2{1-q}>2$ we define a subset $\Gamma^\ka(\uu)$ of the free boundary $\Gamma(\uu)=\p\{|\uu|>0\}$ as 
\begin{align}\label{eq:def-FB}\Gamma^\ka(\uu):=\{x_0\in\Gamma(\uu)\,:\, \uu(x)=O(|x-x_0|^{\xi}) \ \hbox{for some } \lfloor\ka\rfloor<\xi<\ka\}.\end{align} 
Here, the big $O$ is not necessarily  uniform on $\uu$ and $x_0$, and $\lfloor s\rfloor$ is the greatest integer less than $s$, i.e., $s-1\le\lfloor s\rfloor< s$.

\begin{theorem}[Optimal growth estimate]\label{thm:opt-growth}
Let $\uu$ be as in Theorem~\ref{thm:grad-u-holder}. Then there are constants $C>0$ and $r_0>0$, depending only on $n$, $\al$, $M$, $\ka$, $E(\uu,1)$, such that \begin{align*}
    \int_{B_r(x_0)}\left(|\D\uu|^2+|\uu|^{1+q}\right)\le Cr^{n+2\ka-2}
\end{align*}
for $x_0\in \Gamma^\ka(\uu)\cap B_{1/2}$ and $0<r<r_0$.
\end{theorem}
The proof is inspired by the ones for minimizers in \cite{FotShaWei21} and for the case $q=0$ in \cite{DeSJeoSha21}. However, in our case concerning almost minimizers with $\ka>2$, several new technical difficulties arise and the proof is much more complicated, as we have to improve the previous techniques by using approximation by harmonic polynomials and limiting argument.

One implication of Theorem~\ref{thm:opt-growth} is the existence of $\ka$-homogeneous blowups (Theorem~\ref{lem:blowup-exist}). This allows to consider a subset of $\Gamma^\ka(\uu)$, the so-called ``regular" set, using a class of \emph{half-space} solutions
\begin{align*}
  \Hf_{x_0} &:= \{ x\mapsto\beta_{x_0}\max(x\cdot\nu,0)^\ka\mathbf{e} : 
\nu\in\R^n \textrm{ and }
\mathbf{e}\in\R^m \textrm{ are unit vectors}\},\quad x_0\in B_1.
\end{align*}

\begin{definition}[Regular free boundary points]
We say that a point $x_0\in \Gamma^\ka(\uu)$ is a \emph{regular free boundary point} if at least one homogeneous blowup of $\uu$ at $x_0$ belongs to $\Hf_{x_0}$. We denote by $\mathcal{R}_\uu$ the set of all regular free boundary points in $\Gamma(\uu)$ and call it the \emph{regular set}.
\end{definition}

The following is our central result concerning the regularity of the free boundary.

\begin{theorem}[Regularity of the regular set]\label{thm:reg-set}
$\mathcal{R}_\uu$ is a relatively open subset of the free boundary $\Gamma(\uu)$ and locally a $C^{1,\gamma}$-manifold for some $\gamma=\gamma(n,\al, q,\eta)>0$, where $\eta$ is the constant in Theorem~\ref{epi}.
\end{theorem}

The proof is based on the use of the epiperimetric inequality from \cite{FotShaWei21} and follows the general approach in \cite{JeoPet21} and \cite{DeSJeoSha21}: The combination of the monotonicity of Weiss-type energy functional (Theorem~\ref{thm:Weiss}) and the epiperimetric inequality, together with Theorem~\ref{thm:opt-growth}, establishes the geometric decay rate for the Weiss functional. This, in turn, provides us with the rate of convergence of proper rescalings to a blowup, ultimately implying the regularity of $\mathcal R_\uu$.


\subsection{Plan of the paper}
The plan of the paper is as follows.

In Section~\ref{sec:reg-alm-min} we study the regularity properties of almost minimizers. We prove their almost Lipschitz regularity (Theorem~\ref{thm:alm-Lip-reg}) and exploit it to infer the $C^{1,\al/2}$-regularity (Theorem~\ref{thm:grad-u-holder}).

In Section~\ref{sec:weiss} we establish the Weiss-type monotonicity formula (Theorem~\ref{thm:Weiss}), which will play a significant role in the analysis of the free boundary.

Section~\ref{sec:opt-growth} is dedicated to providing the proof of the optimal growth estimates in Theorem~\ref{thm:opt-growth} above.

Section~\ref{sec:nondeg} is devoted to proving the non-degeneracy result of almost minimizers, following the line of \cite{DeSJeoSha21}.

In Section~\ref{sec:blowup} we discuss the homogeneous blowup of almost minimizers at free boundary points, including its existence and properties. In addition, we estimate a decay rate of the Weiss energy, with the help of the epiperimetric inequality.

In Section~\ref{sec:reg-set} we make use of the previous technical tools to establish the $C^{1,\gamma}$-regularity of the regular set (Theorem~\ref{thm:reg-set}).

Finally, in Appendix~\ref{sec:ex} we provide an example of almost minimizers.


\subsection{Notation}\label{subsec:notation}
We introduce here some notations that are used frequently in this paper.\\

\noindent 
$B_r(x_0)$ means the open $n$-dimensional ball of radius $r$, centered at $x_0$, with boundary $\p B_r(x_0)$.\\
$B_r:=B_r(0)$, $\p B_r:=\p B_r(0)$.\\
For a given set, $\nu$ denotes the unit outward normal to the boundary.\\
$\p_\theta\uu:=\D\uu-(\D\uu\cdot\nu)\nu$ is the surface derivative of $\uu$.\\
 For $\uu=(u_1,\cdots,u_m)$, $m\ge1$, we denote $$
\uu^+=(u_1^+,\cdots,u_m^+),\quad \uu^-=(u_1^-,\cdots,u_m^-),\quad\text{where }u_i^\pm=\max\{0,\pm u_i\}.$$
For a domain $D$, we indicate the integral mean value of $\uu$ by $$
\mean{\uu}_D:=\fint_D\uu=\frac1{|D|}\int_D\uu.
$$
In particular, when $D=B_r(x_0)$, we simply write
 $$
\mean{\uu}_{x_0,r}:=\mean{\uu}_{B_r(x_0)}.
$$
$\Gamma(\uu):=\p\{|\uu|>0\}$ is the free boundary of $\uu$.\\
$\Gamma^\ka(\uu):=\{x_0\in\Gamma(\uu)\,:\, \uu(x)=O(|x-x_0|^{\xi}) \ \hbox{for some } \lfloor\ka\rfloor<\xi<\ka\}.$\\
$\lfloor s\rfloor$ is the greatest integer below $s\in\R$, i.e., $s-1\le\lfloor s\rfloor< s$.\\
For $\uu\in W^{1,2}(B_r;\R^m)$ and $0<q<1$, we set $$
E(\uu,r):=\int_{B_r}\left(|\D\uu|^2+|\uu|^{q+1}\right).
$$
For $\al$--H\"older continuous functions $\la_\pm:D\to\R^n$ satisfying $\la_0\le\la_\pm(x)\le\la_1$ (as in Subsection~\ref{subsec:system}), we denote \begin{align*}
f(x_0,\uu)&:=\la_+(x_0)|\uu^+|^{q-1}\uu^+-\la_-(x_0)|\uu^-|^{q-1}\uu^-,\\
    F(x,\uu)&:=\frac1{1+q}\left(\la_+|\uu^+|^{q+1}+\la_-|\uu^-|^{q+1}\right),\\
    F(x_0,\uu)&:=\frac1{1+q}\left(\la_+(x_0)|\uu^+|^{q+1}+\la_-(x_0)|\uu^-|^{q+1}\right),\\
    F(x_0,\uu,\h)&:=\frac1{1+q}\left(\la_+(x_0)\left||\uu^+|^{q+1}-|\h^+|^{q+1}\right|+\la_-(x_0)\left||\uu^-|^{q+1}-|\h^-|^{q+1}\right|\right).
\end{align*}
We fix constants (for $x_0\in B_{1}$) $$
\ka:=\frac2{1-q},\qquad \be_{x_0}=\la_+(x_0)^{\ka/2}(\ka(\ka-1))^{-\ka/2}.
$$
Throughout this paper, a universal constant may depend only on $n$, $\al$, $M$, $\ka$ and $E(\uu,1)$.\\
Below we consider only norms of vectorial functions to $\R^m$, but not those of scalar functions. Thus, for notational simplicity we drop $\R^m$ for spaces of vectorial functions, e.g., $C^1(\R^n)=C^1(\R^n;\R^m)$, $W^{1,2}(B_1)=W^{1,2}(B_1;\R^m)$.


\section{Regularity of almost minimizers}\label{sec:reg-alm-min}
The main result of this section is the $C^{1,\al/2}$ estimates of almost minimizers (Theorem~\ref{thm:grad-u-holder}). The proof is based on the Morrey and Campanato space embeddings, similar to the case of almost minimizers with $q=0$ and $\la_\pm=1$, treated by the authors in \cite{DeSJeoSha21}. We first prove the following concentric ball estimates.

\begin{proposition}\label{prop:alm-min-alm-sub-mean-prop}
Let $\uu$ be an almost minimizer in $B_1$. Then, there are $r_0=r_0(\al,M)\in (0,1)$ and $C_0=C_0(n,M)>1$ such that \begin{align}\label{eq:alm-min-alm-sub-mean-prop}\begin{multlined}
    \int_{B_\rho(x_0)}\left(|\D\uu|^2+F(x_0,\uu)\right)\le C_0\left[\left(\frac\rho r\right)^n+r^\al\right]\int_{B_r(x_0)}\left(|\D\uu|^2+F(x_0,\uu)\right)\\
    +C_0r^{n+\frac2{1-q}(q+1-\al q)}
\end{multlined}\end{align}
for any $B_{r_0}(x_0)\Subset B_1$ and $0<\rho<r<r_0.$
\end{proposition}

\begin{proof}
Without loss of generality, we may assume $x_0=0$. Let $\h$ be a harmonic replacement of $\uu$ in $B_r$, i.e., $\h$ is the vectorial harmonic function with $\h=\uu$ on $\p B_r$. Since $|\h^\pm|^{q+1}$ and $|\D\h|^2$ are subharmonic in $B_r$, we have the following sub-mean value properties: \begin{align}
    \label{eq:sub-mean-value}
    \int_{B_\rho}F(0,\h)\le \left(\frac\rho r\right)^n\int_{B_r}F(0,\h),\quad \int_{B_\rho}|\D \h|^2\le \left(\frac\rho r\right)^n\int_{B_r}|\D\h|^2,\quad 0<\rho<r.
\end{align}
Moreover, notice that since $\h$ is harmonic, $\int_{B_r}\D\h\cdot\D(\uu-\h)=0.$ Combining this with the almost minimizing property of $\uu$, we obtain that for $0<r<r_0(\al,M)$, \begin{align}\label{eq:grad-(u-h)-est}\begin{split}
&\int_{B_r}|\D(\uu-\h)|^2=\int_{B_r}\left(|\D\uu|^2-|\D\h|^2\right)\\
&\le \int_{B_r}\left[Mr^\al|\D\h|^2+2(1+Mr^\al)F(0,\h)-2F(0,\uu)\right]
\\
&=\int_{B_r}\left[Mr^\al|\D\h|^2+2(1+2Mr^\al)(F(0,\h)-F(0,\uu))+2Mr^\al(2F(0,\uu)-F(0,\h))\right]\\
&\le \int_{B_r}\left[Mr^\al|\D\uu|^2+3F(0,\uu,\h)+2Mr^\al(2F(0,\uu)-F(0,\h))\right],
\end{split}\end{align}
where in the last line we have used that $F(0,\h)-F(0,\uu)\le F(0,\uu,\h)$.

We also note that by  Poincar\'e inequality there is $C_1=C_1(n)>0$ such that $$
r^{-2}\int_{B_r}|\uu-\h|^2\le C_1\int_{B_r}|\D(\uu-\h)|^2.
$$
Then, for $\e_1=\frac1{16C_1M}$, \begin{align*}
    &\int_{B_r}\frac1{1+q}\left||\uu^+|^{q+1}-|\h^+|^{q+1}\right|\\
    &\le \int_{B_r}\left(|\uu^+|^q+|\h^+|^q\right)|\uu^+-\h^+|\\
    &\le \int_{B_r}1/4\left(r^{\frac{\al q}{q+1}}|\uu^+|^q\right)^{\frac{q+1}q}+1/4\left(r^{\frac{\al q}{q+1}}|\h^+|^q\right)^{\frac{q+1}q}+C\left(r^{-\frac{\al q}{q+1}}|\uu^+-\h^+|\right)^{q+1}\\
    &=\int_{B_r}r^\al/4(|\uu^+|^{q+1}+|\h^+|^{q+1})+Cr^{-\al q}|\uu-\h|^{q+1}\\
    &\le \int_{B_r}\left[r^\al/4(|\uu^+|^{q+1}+|\h^+|^{q+1})+\e_1\left(r^{-(q+1)}|\uu-\h|^{q+1}\right)^{\frac2{q+1}}+Cr^{(q+1-\al q)\frac2{1-q}}\right]\\
    &= \int_{B_r}\left[r^\al/4(|\uu^+|^{q+1}+|\h^+|^{q+1})+\e_1 r^{-2}|\uu-\h|^2\right]+Cr^{n+\frac2{1-q}(q+1-\al q)}\\
    &\le \int_{B_r}\left[r^\al/4(|\uu^+|^{q+1}+|\h^+|^{q+1})+\e_1 C_1|\D(\uu-\h)|^2\right]+Cr^{n+\frac2{1-q}(q+1-\al q)},
\end{align*}
where in the second inequality we applied Young's inequality.

Similarly, we can get \begin{align*}
    &\int_{B_r}\frac1{1+q}||\uu^-|^{q+1}-|\h^-|^{q+1}|\\
    &\qquad\le \int_{B_r}\left[r^\al/4(|\uu^-|^{q+1}+|\h^-|^{q+1})+\e_1 C_1|\D(\uu-\h)|^2\right]+Cr^{n+\frac2{1-q}(q+1-\al q)},
\end{align*} 
and it follows that \begin{align}
    \label{eq:(u-h)-est}\begin{split}
            &\int_{B_r}F(0,\uu,\h)\\
            &=\int_{B_r}\left(\frac{\la_+(0)}{1+q}\left||\uu^+|^{q+1}-|\h^+|^{q+1}\right|+\frac{\la_-(0)}{1+q}\left||\uu^-|^{q+1}-|\h^-|^{q+1}\right|\right)\\
            &\begin{multlined}\le\int_{B_r}\big[\la_+(0)r^\al/4\left(|\uu^+|^{q+1}+|\h^+|^{q+1}\right)+\la_-(0)r^\al/4\left(|\uu^-|^{q+1}+|\h^-|^{q+1}\right)\\
            +2\e_1 C_1M|\D(\uu-\h)|^2\big]+Cr^{n+\frac2{1-q}(q+1-\al q)}\end{multlined}\\
            &=\int_{B_r}\left[(1+q)r^\al/4(F(0,\uu)+F(0,\h))+2\e_1 C_1M|\D(\uu-\h)|^2\right]+Cr^{n+\frac2{1-q}(q+1-\al q)}.
    \end{split}
\end{align}
From \eqref{eq:grad-(u-h)-est} and \eqref{eq:(u-h)-est}, \begin{align*}
    &\int_{B_r}\left(|\D(\uu-\h)|^2+4F(0,\uu,\h)\right)\\
    &\begin{multlined}\le \int_{B_r}\big[Mr^\al|\D\uu|^2+3F(0,\uu,\h)+Cr^\al F(0,\uu)+\left(1+q-2M\right)r^\al F(0,\h)\\
    +8\e_1 C_1M|\D(\uu-\h)|^2\big]
    +Cr^{n+\frac2{1-q}(q+1-\al q)}
    \end{multlined}\\
    &\begin{multlined}\le \int_{B_r}\left[Cr^\al\left(|\D\uu|^2+F(0,\uu)\right)+3F(0,\uu,\h)+\frac12|\D(\uu-\h)|^2\right]\\
    +Cr^{n+\frac2{1-q}(q+1-\al q)},\end{multlined}
\end{align*}
which gives \begin{align*}\begin{split}\begin{multlined}
    \int_{B_r}\left(\frac12|\D(\uu-\h)|^2+F(0,\uu,\h)\right)
    \le Cr^\al\int_{B_r}\left(|\D\uu|^2+F(0,\uu)\right)\\+Cr^{n+\frac2{1-q}(q+1-\al q)}.\end{multlined}\end{split}
\end{align*}
Thus \begin{align}\label{eq:(u-h)-diff-est}\begin{split} &\int_{B_r}\left(|\D(\uu-\h)|^2+F(0,\uu,\h)\right)\\
    &\qquad\le 2\int_{B_r}\left(\frac12|\D(\uu-\h)|^2+F(0,\uu,\h)\right)\\
    &\qquad\le Cr^\al\int_{B_r}\left(|\D\uu|^2+F(0,\uu)\right)+Cr^{n+\frac2{1-q}(q+1-\al q)}.\end{split}
\end{align}
Now, by combining \eqref{eq:sub-mean-value} and \eqref{eq:(u-h)-diff-est}, we obtain that for $0<\rho<r<r_0$,\begin{align*}
    &\int_{B_\rho}\left(|\D\uu|^2+F(0,\uu)\right)\\
    &\le 2\int_{B_\rho}\left(|\D\h|^2+F(0,\h)\right)+2\int_{B_\rho}\left(|\D(\uu-\h)|^2+F(0,\uu,\h)\right)\\
    &\le 2\left(\frac\rho r\right)^n\int_{B_r}\left(|\D\h|^2+F(0,\h)\right)+2\int_{B_\rho}\left(|\D(\uu-\h)|^2+F(0,\uu,\h)\right)\\
    &\le 4\left(\frac\rho r\right)^n\int_{B_r}\left(|\D\uu|^2+F(0,\uu)\right)+6\int_{B_r}\left(|\D(\uu-\h)|^2+F(0,\uu,\h)\right)\\
    &\le C\left[\left(\frac\rho r\right)^n+r^\al\right]\int_{B_r}\left(|\D\uu|^2+F(0,\uu)\right)+Cr^{n+\frac2{1-q}(q+1-\al q)}.\qedhere
\end{align*}

\end{proof}

From here, we deduce the almost Lipschitz regularity of almost minimizers with the help of the following lemma, whose proof can be found in \cite{HanLin97}.

\begin{lemma}\label{lem:HL}
  Let $r_0>0$ be a positive number and let
  $\vp:(0,r_0)\to (0, \infty)$ be a nondecreasing function. Let $a$,
  $\beta$, and $\gamma$ be such that $a>0$, $\gamma >\beta >0$. There
  exist two positive numbers $\e=\e(a,\gamma,\beta)$,
  $c=c(a,\gamma,\beta)$ such that, if
$$
\vp(\rho)\le
a\Bigl[\Bigl(\frac{\rho}{r}\Bigr)^{\gamma}+\e\Bigr]\vp(r)+b\, r^{\be}
$$ for all $\rho$, $r$ with $0<\rho\leq r<r_0$, where $b\ge 0$,
then one also has, still for $0<\rho<r<r_0$,
$$
\vp(\rho)\le
c\Bigl[\Bigl(\frac{\rho}{r}\Bigr)^{\be}\vp(r)+b\rho^{\be}\Bigr].
$$
\end{lemma}

\begin{theorem}\label{thm:alm-Lip-reg}
Let $\uu$ be an almost minimizer in $B_1$. Then $\uu\in C^{0,\si}(B_1)$ for all $0<\si<1$. Moreover, for any $K\Subset B_1$, \begin{align}\label{eq:alm-Lip-reg}
\|\uu\|_{C^{0,\si}(K)}\le C\left(E(\uu,1)^{1/2}+1\right)
\end{align}
with $C=C(n,\al,M,\si,K)$.
\end{theorem}

\begin{proof}
For given $K\Subset B_1$ and $x_0\in K$, take $\de=\de(n,\al,M,\si,K)>0$ such that $\de<\min\{r_0,\dist(K,\p B_1)\}$ and $\de^\al\le \e(C_0,n,n+2\si-2)$, where $r_0=r_0(\al,M)$ and $C_0=C_0(n,M)$ are as in Proposition~\ref{prop:alm-min-alm-sub-mean-prop} and $\e=\e(C_0,n,n+2\si-2)$ is as in Lemma~\ref{lem:HL}. Then, by \eqref{eq:alm-min-alm-sub-mean-prop}, for $0<\rho<r<\de$, \begin{align*}
    \int_{B_\rho(x_0)}\left(|\D\uu|^2+F(x_0,\uu)\right)\le C_0\left[\left(\frac\rho r\right)^n+\e\right]\int_{B_r(x_0)}\left(|\D\uu|^2+F(x_0,\uu)\right)+C_0r^{n+2\si-2}.
\end{align*}
By applying Lemma~\ref{lem:HL}, we obtain \begin{align*}\begin{multlined}
    \int_{B_\rho(x_0)}\left(|\D \uu|^2+F(x_0,\uu)\right)
    \le C\left[\left(\frac\rho r\right)^{n+2\si-2}\int_{B_r(x_0)}\left(|\D\uu|^2+F(x_0,\uu)\right)+\rho^{n+2\si-2}\right].
\end{multlined}\end{align*}
Taking $r\nearrow\de$, we get \begin{align}\label{eq:alm-min-Morrey-est}
\int_{B_\rho(x_0)}\left(|\D\uu|^2+F(x_0,\uu)\right)\le C(n,\al,M,\si,K)\left(E(\uu,1)+1\right)\rho^{n+2\si-2}
\end{align}
for $0<\rho<\de$. In particular, we have $$
\int_{B_\rho(x_0)}|\D\uu|^2\le C(n,\al,M,\si,K)\left(E(\uu,1)+1\right)\rho^{n+2\si-2},
$$
and by Morrey space embedding we conclude $\uu\in C^{0,\si}(K)$ with \begin{equation*}
\|\uu\|_{C^{0,\si}(K)}\le C(n,\al,M,\si,K)\left(E(\uu,1)^{1/2}+1\right).
\qedhere\end{equation*}
\end{proof}

We now prove $C^{1,\al/2}$-regularity of almost minimizers by using their almost Lipschitz estimates above.

\begin{proof}[Proof of Theorem~\ref{thm:grad-u-holder}]
For $K\Subset B_1$, fix a small $r_0=r_0(n,\al,M,K)>0$ to be chosen later. Particularly, we ask $r_0<\dist(K,\p B_1)$. For $x_0\in K$ and $0<r<r_0$, let $\h\in W^{1,2}(B_r(x_0))$ be a harmonic replacement of $\uu$ in $B_r(x_0)$. Then, by \eqref{eq:(u-h)-diff-est} and \eqref{eq:alm-min-Morrey-est} with $\si=1-\al/4\in (0,1)$, \begin{align}\begin{split}\label{eq:grad-(u-h)-est-2}
    \int_{B_r(x_0)}|\D(\uu-\h)|^2&\le Cr^\al\int_{B_r(x_0)}\left(|\D\uu|^2+F(x_0,\uu)\right)+Cr^{n+\frac2{1-q}(q+1-\al q)}\\
    &\le C\left(E(\uu,1)+1\right)r^{n+\al/2}+Cr^{n+2}\\
    &\le C(n,\al,M,K)\left(E(\uu,1)+1\right)r^{n+\al/2}
\end{split}\end{align}
for $0<r<r_0$. Note that since $\h$ is harmonic in $B_r(x_0)$, for $0<\rho<r$ $$
\int_{B_\rho(x_0)}|\D\h-\mean{\D\h}_{x_0,\rho}|^2\le \left(\frac\rho r\right)^{n+2}\int_{B_r(x_0)}|\D\h-\mean{\h}_{x_0,r}|^2.
$$
Moreover, by Jensen's inequality, \begin{align*}
    &\int_{B_\rho(x_0)}|\D\uu-\mean{\D\uu}_{x_0,\rho}|^2\\
    &\qquad\le 3\int_{B_\rho(x_0)}|\D\h-\mean{\D\h}_{x_0,\rho}|^2+|\D(\uu-\h)|^2+|\mean{\D(\uu-\h)}_{x_0,\rho}|^2\\
    &\qquad\le 3\int_{B_\rho(x_0)}|\D\h-\mean{\D\h}_{x_0,\rho}|^2+6\int_{B_\rho(x_0)}|\D(\uu-\h)|^2,
\end{align*}
and similarly, $$
\int_{B_r(x_0)}|\D\h-\mean{\D\h}_{x_0,r}|^2\le 3\int_{B_r(x_0)}|\D \uu-\mean{\D\uu}_{x_0,r}|^2+6\int_{B_r(x_0)}|\D(\uu-\h)|^2.
$$
Now, we use the inequalities above to obtain \begin{align*}
    &\int_{B_\rho(x_0)}|\D\uu-\mean{\D\uu}_{x_0,\rho}|^2\\
    &\qquad\le 3\int_{B_\rho(x_0)}|\D\h-\mean{\D\h}_{x_0,\rho}|^2+6\int_{B_\rho(x_0)}|\D(\uu-\h)|^2\\
    &\qquad\le 3\left(\frac\rho r\right)^{n+2}\int_{B_r(x_0)}|\D\h-\mean{\D\h}_{x_0,r}|^2+6\int_{B_\rho(x_0)}|\D(\uu-\h)|^2\\
    &\qquad\le 9\left(\frac\rho r\right)^{n+2}\int_{B_r(x_0)}|\D \uu-\mean{\D \uu}_{x_0,r}|^2+24\int_{B_r(x_0)}|\D(\uu-\h)|^2\\
    &\qquad\le 9\left(\frac\rho r\right)^{n+2}\int_{B_r(x_0)}|\D\uu-\mean{\D\uu}_{x_0,r}|^2+C\left(E(\uu,1)+1\right)r^{n+\al/2}.
\end{align*}
Next, we apply Lemma~\ref{lem:HL} to get \begin{align*}
    \begin{multlined}\int_{B_\rho(x_0)}|\D\uu-\mean{\D\uu}_{x_0,\rho}|^2\le C\left(\frac\rho r\right)^{n+\al/2}\int_{B_r(x_0)}|\D\uu-\mean{\D\uu}_{x_0,r}|^2\\
    +C\left(E(\uu,1)+1\right)\rho^{n+\al/2}\end{multlined}
\end{align*}
for $0<\rho<r<r_0$. Taking $r\nearrow r_0$, we have \begin{align*}
\int_{B_\rho(x_0)}|\D\uu-\mean{\D\uu}_{x_0,\rho}|^2\le C\left(E(\uu,1)+1\right)\rho^{n+\al/2}.
\end{align*}
By Campanato space embedding, we obtain $\D\uu\in C^{0,\al/4}(K)$ with $$
\|\D\uu\|_{C^{0,\al/4}(K)}\le C(E(\uu,1)^{1/2}+1).
$$
In particular, we have $$
\|\D\uu\|_{L^\infty(K)}\le C(E(\uu,1)^{1/2}+1)
$$
for any $K\Subset B_1$. With this estimate and \eqref{eq:alm-Lip-reg}, we can improve \eqref{eq:grad-(u-h)-est-2}: \begin{align*}
    \int_{B_r(x_0)}|\D(\uu-\h)|^2&\le Cr^\al\int_{B_r(x_0)}\left(|\D\uu|^2+F(x_0,\uu)\right)+Cr^{n+2}\\
    &\le C\left(E(\uu,1)+1\right)r^{n+\al}+Cr^{n+2}\\
    &\le C\left(E(\uu,1)+1\right)r^{n+\al},
\end{align*}
and by repeating the process above we conclude that $\D\uu\in C^{1,\al/2}(K)$ with \begin{equation*}
\|\D\uu\|_{C^{0,\al/2}(K)}\le C(n,\al,M,K)(E(\uu,1)^{1/2}+1).
\qedhere\end{equation*}
\end{proof}


\section{Weiss-type monotonicity formula}\label{sec:weiss}
In the rest of the paper we study the free boundary of almost minimizers. This section is devoted to proving Weiss-type monotonicity formula, which is one of the most important tools in our study of the free boundary. This result is obtained from comparison with $\ka$-homegeneous replacements, following the idea for the one in the case $q=0$ in \cite{DeSJeoSha21}.

\begin{theorem}[Weiss-type monotonicity formula]\label{thm:Weiss}
Let $\uu$ be an almost minimizer in $B_1$. For $\ka=\frac2{1-q}>2$ and $x_0, x_1\in B_{1/2}$, set $$
W(\uu,x_0,x_1,t):=\frac{e^{at^\al}}{t^{n+2\ka-2}}\left[\int_{B_t(x_0)}\left(|\D\uu|^2+2F(x_1,\uu)\right)-\frac{\ka(1-bt^\al)}t\int_{\p B_t(x_0)}|\uu|^2\right],
$$
with $$
a=\frac{M(n+2\ka-2)}\al,\quad b=\frac{M(n+2\ka)}\al.
$$
Then, for $0<t<t_0(n,\al,\ka,M)$, $$
\frac{d}{dt}W(\uu,x_0,x_1,t)\ge\frac{e^{at^\al}}{t^{n+2\ka-2}}\int_{\p B_t(x_0)}\left|\p_\nu\uu-\frac{\ka(1-bt^\al)}t\uu\right|^2.
$$
In particular, $W(\uu,x_0,x_1,t)$ is nondecreasing in $t$ for $0<t<t_0$.
\end{theorem}

\begin{proof}
We follow the argument in Theorem 5.1 in \cite{JeoPet21}. Without loss of generality, we may assume $x_0=0$. Then, for $0<t<1/2$, define the $\ka$-homogeneous replacement of $\uu$ in $B_t$
$$
\ww(x):=\left(\frac{|x|}t\right)^\ka\uu\left(t\frac{x}{|x|}\right),\quad x\in B_t.
$$
Note that $\ww$ is homogeneous of degree $\ka$ in $B_t$ and coincides with $\uu$ on $\p B_t$, that is a valid competitor for $\uu$ in $B_t$. We compute \begin{align*}
    \int_{B_t}|\D\ww|^2&=\int_{B_t}\left(\frac{|x|}t\right)^{2\ka-2}\left|\frac\ka t\uu\left(t\frac{x}{|x|}\right)\frac{x}{|x|}+\D\uu\left(t\frac{x}{|x|}\right)-\D\uu\left(t\frac{x}{|x|}\right)\cdot\frac{x}{|x|}\frac{x}{|x|}\right|^2\\
    &=\int_0^t\int_{\p B_r}\left(\frac rt\right)^{2\ka-2}\left|\frac \ka tu\left(t\frac xr\right)\nu-\left(\D\uu\left(t\frac xr\right)\nu\right)\nu+\D\uu\left(t\frac xr\right)\right|^2\,dS_x\,dr\\
    &=\int_0^t\int_{\p B_t}\left(\frac rt\right)^{n+2\ka-3}\left|\frac\ka t\uu\nu-(\p_\nu\uu)\nu+\D\uu\right|^2\,dS_x\,dr\\
    &=\frac{t}{n+2\ka-2}\int_{\p B_t}\left|\D\uu-(\p_\nu\uu)\nu+\frac \ka t\uu\nu\right|^2\,dS_x\\
    &=\frac{t}{n+2\ka-2}\int_{\p B_t}\left(|\D\uu|^2-|\p_\nu\uu|^2+\left(\frac\ka t\right)^2|\uu|^2\right).
\end{align*}
Moreover, from \begin{align*}
    \int_{B_t}|\ww^\pm|^{q+1}&=\int_0^t\int_{\p B_r}\left(\frac rt\right)^{2\ka-2}\left|\uu^\pm\left(\frac trx\right)\right|^{q+1}\,dS_x\,dr\\
    &=\int_0^t\int_{\p B_t}\left(\frac rt\right)^{n+2\ka-3}|\uu^\pm|^{q+1}\,dS_x\,dr\\
    &=\frac{t}{n+2\ka-2}\int_{\p B_t}|\uu^\pm|^{q+1},
\end{align*}
we also have $$
\int_{B_t}F(x_1,\ww)=\frac t{n+2\ka-2}\int_{\p B_t}F(x_1,\uu).
$$
Combining those computations with the almost minimizing property of $\uu$, we get \begin{align*}
    &(1-Mt^\al)\int_{B_t}\left(|\D\uu|^2+2F(x_1,\uu)\right)\\
    &\qquad\le \frac1{1+Mt^\al}\int_{B_t}\left(|\D\uu|^2+2F(x_1,\uu)\right)\le \int_{B_t}\left(|\D\ww|^2+2F(x_1,\ww)\right)\\
    &\qquad=\frac t{n+2\ka-2}\int_{\p B_t}\left(|\D\uu|^2-|\p_\nu\uu|^2+\left(\frac\ka t\right)^2|\uu|^2+2F(x_1,\uu)\right).
\end{align*}
This gives \begin{align}\label{eq:Weiss-alm-min-ineq}\begin{split}
    &\frac{d}{dt}\left(e^{at^\al}t^{-n-2\ka+2}\right)\int_{B_t}\left(|\D\uu|^2+2F(x_1,\uu)\right)\\
    &\qquad=-(n+2\ka-2)e^{at^\al}t^{-n-2\ka+1}(1-Mt^\al)\int_{B_t}\left(|\D\uu|^2+2F(x_1,\uu)\right)\\
    &\qquad\ge -e^{at^\al}t^{-n-2\ka+2}\int_{\p B_t}\left(|\D\uu|^2-|\p_\nu\uu|^2+\left(\frac\ka{t}\right)^2|\uu|^2+2F(x_1,\uu)\right).
\end{split}\end{align}
Note that we can write $$
W(\uu,0,x_1,t)=e^{at^\al}t^{-n-2\ka+2}\int_{B_t}\left(|\D\uu|^2+2F(x_1,\uu)\right)-\psi(t)\int_{\p B_t}|\uu|^2,
$$
where $$
\psi(t)=\frac{\ka e^{at^\al}(1-bt^\al)}{t^{n+2\ka-1}}.
$$
Then, using \eqref{eq:Weiss-alm-min-ineq}, we obtain \begin{align*}
    &\frac d{dt}W(\uu,0,x_1,t)\\
    &=\frac{d}{dt}\left(e^{at^\al}t^{-n-2\ka+2}\right)\int_{B_t}\left(|\D\uu|^2+2F(x_1,\uu)\right)+e^{at^\al}t^{-n-2\ka+2}\int_{\p B_t}\left(|\D\uu|^2+2F(x_1,\uu)\right)\\
    &\qquad-\psi'(t)\int_{\p B_t}|\uu|^2-\psi(t)\int_{\p B_t}\left(2\uu\p_\nu\uu+\frac{n-1}t|\uu|^2\right)\\
    &\ge e^{at^\al}t^{-n-2\ka+2}\int_{\p B_t}|\p_\nu\uu|^2-2\psi(t)\int_{\p B_t}\uu\p_\nu\uu\\
    &\qquad-\left(\ka^2e^{at^\al}t^{-n-2\ka}+\psi'(t)+(n-1)\frac{\psi(t)}t\right)\int_{\p B_t}|\uu|^2.
\end{align*}
To simplify the last term, we observe that $\psi(t)$ satisfies the inequality $$
-\frac{e^{at^\al}}{t^{n+2\ka-2}}\left(\ka^2e^{at^\al}t^{-n-2\ka}+\psi'(t)+(n-1)\frac{\psi(t)}t\right)-\psi(t)^2\ge 0
$$
for $0<t<t_0(n,\al,\ka,M)$. Indeed, by a direct computation, we can see that the inequality above is equivalent to $$
2\al^2-M(n+2\ka)\left[(n+2\ka)(\ka-\al)+2\al\right]t^\al\ge 0,
$$
which holds for $0<t<t_0(n,\al,\ka,M)$. Therefore, we conclude that \begin{align*}
    \frac{d}{dt}W(\uu,0,x_1,t)&\ge e^{at^\al}t^{-n-2\ka+2}\int_{\p B_t}|\p_\nu\uu|^2-2\psi(t)\int_{\p B_t}\uu\p_\nu\uu\\
    &\qquad+e^{-at^\al}t^{n+2\ka-2}\psi(t)^2\int_{\p B_t}|\uu|^2\\
    &=e^{at^\al}t^{-n-2\ka+2}\int_{\p B_t}\left|\p_\nu\uu-e^{-at^\al}t^{n+2\ka-2}\psi(t)\uu\right|^2\\
    &=\frac{e^{at^\al}}{t^{n+2\ka-2}}\int_{\p B_t}\left|\p_\nu\uu-\frac{\ka(1-bt^\al)}t\uu\right|^2.\qedhere
\end{align*}
\end{proof}


\section{Growth estimates}\label{sec:opt-growth}
In this section we prove the optimal growth of almost minimizers at the free boundary (Theorem~\ref{thm:opt-growth}).

We will divide our proof into two cases: $$
\text{either $\ka\not\in\N$ or $\ka\in\N$.}
$$
The proof for the first case $\ka\not\in\N$ is given in Lemma~\ref{lem:growth-est-ka-noninteger}, and the one for the second case $\ka\in\N$ can be found in Lemma~\ref{lem:v-p-opt-growth-integer} and Remark~\ref{rmk:opt-growth-integer}.

We start the proof with an auxiliary result on a more general class of almost minimizers.

\begin{lemma}
\label{lem:alm-min-decay}
For $0<a_0\le 1$, $0<b_0\le 1$ and $z_0\in B_{1/2}$, we define $G(z,\uu):=a_0F(b_0z+z_0,\uu)$ and let $\uu$ be an almost minimizer in $B_1$ of functionals \begin{align}\label{eq:alm-min-ftnal}
\int_{B_r(z)}\left(|\D\uu|^2+2G(z,\uu)\right)\,dx,\qquad B_r(z)\Subset B_1,
\end{align}
with a gauge function $\omega(r)=Mr^\al.$ If $\uu(x)=O(|x-x_0|^\mu)$ for some $1\le\mu\le\ka$, $\mu\not\in\N$, and $x_0\in B_{1/2}$, then \begin{align}\label{eq:alm-min-decay-est}
|\uu(x)|\le C_\mu|x-x_0|^\mu,\qquad |\D\uu(x)|\le C_\mu|x-x_0|^{\mu-1},\quad |x-x_0|\le r_\mu
\end{align}
with constants $C_\mu$ and $r_\mu$ depending only on $E(\uu,1)$, $n$, $\al$, $M$, $\mu$. As before, the $O(\cdot)$ notation does not necessarily mean the uniform estimate.
\end{lemma}

\begin{proof}
We can write $G(z,\uu)=\frac1{1+q}\left(\tilde\la_+(z)|\uu^+|^{q+1}+\tilde\la_-(z)|\uu^-|^{q+1}\right)$ for $\tilde\la_\pm(z)=a_0\la_\pm(b_0z+z_0)$, which means that $\uu$ is an almost minimizer of the energy functional \eqref{eq:energy-ftnal} with variable coefficients $\tilde\la_\pm$. In the previous sections we have proved that almost minimizers with \eqref{eq:M} satisfies the $C^{1,\al/2}$-estimate \eqref{eq:grad-u-holder}. $\uu$ also satisfies \eqref{eq:M} but $1/\tilde\la_0\le M$ for the lower bound $\tilde\la_0$ of $\la_\pm$, since $a_0<1$. One can check, however, that in the proofs towards \eqref{eq:grad-u-holder} the bound $1/\la_0\le M$ in \eqref{eq:M} is used only to get the estimate for $\frac{\la_\pm(x)}{\la_\pm(x_0)}$ in \eqref{eq:lambda-ratio} (when rewriting the almost minimizing property with variable coefficients \eqref{eq:alm-min} to frozen coefficients \eqref{eq:alm-min-frozen}). Due to cancellation $\frac{\tilde\la_\pm(x)}{\tilde\la_\pm(x_0)}=\frac{\la_\pm(b_0x+z_0)}{\la_\pm(b_0x_0+z_0)}$ satisfies \eqref{eq:lambda-ratio}, thus we can apply Theorem~\ref{thm:grad-u-holder} to $\uu$ to obtain the uniform estimate
$$
\|\uu\|_{C^{1,\al/2}(B_{1/2})}\le C(n,\al,M)\left(E(\uu,1)^{1/2}+1\right).
$$
In view of this estimate, the statement of Lemma~\ref{lem:alm-min-decay} holds for $\mu=1$. 

Now we assume that the statement holds for $1\le\mu<\ka$ and prove that it holds for $\mu+\de\le\ka$ with $\de<\al'/2$, $\al'=\al'(n,q)\le \al$ small enough, and $\mu+\de\not\in\N$. This will readily imply Lemma~\ref{lem:alm-min-decay} by bootstrapping.

First, we claim that \eqref{eq:alm-min-decay-est} implies that there exist constants $C_0>0$ and $r_0>0$, depending only on $E(\uu,1)$, $n$, $\al$, $M$, $\mu$, $\de$, such that for any $r\le r_0$\begin{align}
    \label{eq:alm-min-poly-diff-est}\begin{split}
    &\text{there is a harmonic polynomial }\pp^r\text{ of degree }s:=\lfloor\mu+\de\rfloor\in(\mu+\de-1,\mu+\de)\\
    &\text{satisfying }\fint_{B_r}|\D(\uu-\pp^r)|^2\le C_0r^{2(\mu+\de-1)}.
\end{split}\end{align}
We will prove \eqref{eq:alm-min-poly-diff-est} later, and at this moment assume that it is true. Then, by Poincar\'e inequality (up to possibly modifying $\pp^r$ by a constant and choosing $C_0$ larger),
$$
\fint_{B_r}|\uu-\pp^r|^2\le C_0r^{2(\mu+\de)}.
$$
By a standard limiting argument, using that $s<\mu+\de$, we obtain that for a limiting polynomial $\overline{\pp}$, and for all $r\le r_0$,
\begin{align*}
    &\fint_{B_r}|\uu-\overline{\pp}|^2\le \overline{C}r^{2(\mu+\de)},\\
    &\fint_{B_r}|\D(\uu-\overline{\pp})|^2\le \overline{C}r^{2(\mu+\de-1)}.
\end{align*}
From these estimates, under the assumption $\uu(x)=O(|x|^{\mu+\de})$ we deduce $\overline{\pp}\equiv0$, and obtain that for all $r\le r_0$ \begin{align}
    \label{eq:alm-min-decay-est-2}
    \fint_{B_r}|\uu|^2\le\overline{C}r^{2(\mu+\de)},\qquad \fint_{B_r}|\D\uu|^2\le\overline{C}r^{2(\mu+\de-1)}.
\end{align}
On the other hand, using $\mu+\de\le\ka=\frac2{1-q}$, one can easily see that the rescalings $\vv(x):=\frac{\uu(rx)}{r^{\mu+\de}}$, $0<r\le r_0$, are almost minimizers of the functional \eqref{eq:alm-min-ftnal} with $G(z,\vv)=r^{2-(1-q)(\mu+\de)}F(rz,\vv)$ and a gauge function $\omega_r(\rho)=M(r\rho)^\al$. This, together with \eqref{eq:alm-min-decay-est-2}, implies that the $C^{1,\al/2}$-estimates of $\vv$ are uniformly bounded, independent of $r$. This readily gives the desired estimates \eqref{eq:alm-min-decay-est} for $\mu+\de$ $$
|\uu|\le C_{\mu+\de}|x|^{\mu+\de},\qquad|\D\uu|\le C_{\mu+\de}|x|^{\mu+\de-1}.
$$
We are now left with the proof of \eqref{eq:alm-min-poly-diff-est}. To this aim, let $\h$ be the harmonic replacement of $\uu$ in $B_r$. Note that $\h$ minimizes the Dirichlet integral and attains its maximum on $\p B_r$. Combining this with the almost-minimality of $\uu$ and \eqref{eq:alm-min-decay-est} yields
\begin{align}\label{uh}\begin{split}\fint_{B_r}|\nabla (\uu-\h)|^2 & = \fint_{B_r}|\D\uu|^2-\fint_{B_r}|\D\h|^2\\
&\leq Mr^\alpha \fint_{B_r} |\nabla \h|^2 + 2(1+Mr^\alpha)\fint_{B_r} G(0,\h) - 2\fint_{B_r} G(0,\uu)\\
 &\leq C_\mu^2r^{\alpha+2(\mu-1)} + 2C_\mu r^{(q+1)\mu} \leq3 C_\mu^2 r^{\alpha'+2(\mu-1)}.\end{split}\end{align} 
Here, the last inequality holds for $\alpha' \leq \alpha$ small enough since $2(\mu-1)< (q+1)\mu$.

Now, in order to prove \eqref{eq:alm-min-poly-diff-est}, as in standard Campanato Type estimates, it suffices to show that if \eqref{eq:alm-min-poly-diff-est} holds for $r$, then for a fixed constant $\rho$ small enough,
$$\text{$\exists$ $\pp^{r\rho}$ harmonic polynomial of degree $s=\lfloor\mu+\delta\rfloor$ such that }$$ \begin{equation}\label{pp} 
 \fint_{B_{\rho r}} |\nabla (\uu-\pp^{r\rho})|^2 \leq C_0 (\rho r)^{2(\mu+\delta-1)}.
\end{equation} 
Indeed, since $\h-\pp^r$ is harmonic, there exists a harmonic polynomial $\hat\pp^{\rho}$ of degree $s$ such that\begin{align}\label{hpp}
\fint_{B_{\rho r}} |\nabla (\h-\pp^r-\hat\pp^{\rho})|^2& \leq C \rho ^{2s}\fint_{B_{ r}} |\nabla (\h-\pp^r)|^2\\
\nonumber &\leq C \rho ^{2s}\fint_{B_{ r}} |\nabla (\uu-\pp^r)|^2 \leq C C_0\rho^{2s}r^{2(\mu+\delta-1)} \\ \nonumber &\leq \frac{C_0}{4} (\rho r)^{2(\mu+\delta-1)}\end{align}
as long as $\rho$ is small enough, given that $s> \mu+\delta -1.$ To justify the first inequality, notice that if $\ww$ is harmonic in $B_1$ and $\qq$ is the tangent polynomial to $\ww$ at $0$ of degree $s-1$ then 
$$\fint_{B_\rho} |\ww-\qq|^2 \leq \|\ww-\qq\|^2_{L^\infty(B_\rho)} \leq C \rho^{2s} \|\ww\|^2_{L^\infty(B_{3/4})} \leq C \rho^{2s} \fint_{B_1} \ww^2.$$
Thus, we are applying this inequality to $\ww=\p_i (\h-\pp^r)$ and $\qq=\p_i \hat\pp^\rho$, $1\le i\le n$, with $\hat\pp^\rho$ the tangent polynomial to $\h-\pp^r$ at 0 of degree $s$. The second inequality in \eqref{hpp} follows from the fact that $\h$ is the harmonic replacement of $\uu$ in $B_r$.

From \eqref{uh} for this specific $\rho$ for which \eqref{hpp} holds, we obtain that
$$\fint_{B_{\rho r}} |\nabla (\uu-\h)|^2\leq \rho^{-n}\fint_{B_r}|\D(\uu-\h)|^2\le \rho^{-n -\alpha'-2(\mu-1)}3 C_\mu^2 (\rho r)^{\alpha'+2(\mu-1)}.$$
Combining this inequality with \eqref{hpp}, since $\delta<\alpha'/2$, we obtain the desired claim with $\pp^{r\rho}= \pp^r + \hat\pp^\rho$, as long as $C_0 \geq 12C^2_\mu  \rho^{-n -\alpha'-2(\mu-1)}.$
\end{proof}

Now we prove the optimal growth at free boundary points (Theorem~\ref{thm:opt-growth}) when $\ka\not\in\mathbb N$.

\begin{lemma}
\label{lem:growth-est-ka-noninteger}
Let $\uu\in W^{1,2}(B_1)$ be an almost minimizer in $B_1$ and $\ka\not\in\N$. Then, there exist $C>0$ and $r_0>0$, depending on $n$, $\al$, M, $\ka$, $E(\uu,1)$, such that $$
\sup_{B_r(x_0)}\left(\frac{|\uu|}{r^{\ka}}+\frac{|\D\uu|}{r^{\ka-1}}\right)\le C ,
$$
for any $x_0\in\Gamma^\ka(\uu)\cap B_{1/2}$ and $0<r<r_0.$\end{lemma}

\begin{proof}

We first prove the weaker version of Lemma~\ref{lem:growth-est-ka-noninteger} by allowing the constants $C$ and $r_0$ to depend on the points $x_0\in\Gamma^\ka(\uu)\cap B_{1/2}$ as well. That is, for each $x_0\in\Gamma^\ka(\uu)\cap B_{1/2}$, \begin{align}
    \label{eq:growth-est-ptwise}
    \sup_{B_r(x_0)}\left(\frac{|\uu|}{r^{\ka}}+\frac{|\D\uu|}{r^{\ka-1}}\right)\le C_{x_0},\quad 0<r<r_{x_0},
\end{align}
where $C_{x_0}$ and $r_{x_0}$ depend on $n$, $\al$, $M$, $\ka$, $E(\uu,1)$ and $x_0$.

To show this weaker estimate \eqref{eq:growth-est-ptwise}, we assume to the contrary there is a point $x_0\in \Gamma^\ka(\uu)\cap B_{1/2}$ and a sequence of positive radii $\{r_j\}^\infty_{j=1}\subset(0,1)$, $r_j\searrow0$, such that \begin{align*}
    \sup_{B_{r_j}(x_0)}\left(\frac{|\uu|}{r_j^{\ka}}+\frac{|\D\uu|}{r_j^{\ka-1}}\right)=j,\quad \sup_{B_r(x_0)}\left(\frac{|\uu|}{r^{\ka}}+\frac{|\D\uu|}{r^{\ka-1}}\right)\le j\quad\text{for any } r_j\le r\le 1/4.
\end{align*}
Define the function $$
\wuuj(x):=\frac{\uu(r_jx+x_0)}{jr_j^{\ka}},\quad x\in B_{\frac1{4r_j}}.
$$
Then
 $$
\sup_{B_1}\left(|\wuuj|+|\D\wuuj|\right)=1\quad\text{and } \sup_{B_R}\left(\frac{|\wuuj|}{R^{\ka}}+\frac{|\D\wuuj|}{R^{\ka-1}}\right)\le 1\quad\text{for any }1\le R\le \frac1{4r_j}.
$$
Now we claim that there exists a harmonic function $\wuuz\in C^1_{\loc}(\R^n)$ such that over a subsequence $$
\wuuj\to\wuuz\quad\text{in }C^1_{\loc}(\R^n).
$$
Indeed, for a fixed $R>1$ and a ball $B_\rho(z)\subset B_R$, we have $$
\int_{B_\rho(z)}\left(|\D\wuuj|^2+2F_j(z,\wuuj)\right)=\frac1{j^2r_j^{n+2\ka-2}}\int_{B_{r_j\rho}(r_jz+x_0)}\left(|\D\uu|^2+2F(r_jz+x_0,\uu)\right)
$$
for $F_j(z,\wuuj):=\frac1{j^{1-q}}F(r_jz+x_0,\wuuj)=\frac1{1+q}\left((\la_j)_+(z)|(\wuuj)^+|^{q+1}+(\la_j)_-(z)|(\wuuj)^-|^{q+1}\right)$, where $(\la_j)_\pm(z)=\frac1{j^{1-q}}\la_\pm(r_jz+x_0)$. Since each $\wuuj$ is an almost minimizer of functional \eqref{eq:alm-min-ftnal} with gauge function $\omega_j(\rho)\le M(r_j\rho)^\al\le M\rho^\al$, we can apply Theorem~\ref{thm:grad-u-holder} to $\wuuj$ to obtain
\begin{align*}
    \|\wuuj\|_{C^{1,\al/2}(\overline{B_{R/2}})}\le C(n,\al,M,R)\left(E(\wuuj,1)^{1/2}+1\right)\le C(n,\al,M,R).
\end{align*}
This implies that up to a subsequence, $$
\wuuj\to\wuuz\quad\text{in }C^1(B_{R/2}).
$$
By letting $R\to\infty$ and using Cantor's diagonal argument, we further have $$
\wuuj\to\wuuz\quad\text{in }C^1_{\loc}(\R^n).
$$
To show that $\wuuz$ is harmonic, we fix $R>1$ and observe that for the harmonic replacement $\h_j$ of $\wuuj$ in $B_R$, \begin{align}\label{eq:alm-min-prop}
\int_{B_R}\left(|\D\wuuj|^2+\frac2{j^{1-q}}F(x_0,\wuuj)\right)\le \left(1+M(r_jR)^\al\right)\int_{B_R}\left(|\D\h_j|^2+\frac2{j^{1-q}}F(x_0,\h_j) \right).
\end{align}
From the global estimates of harmonic function $\h_j$ $$
\|\h_j\|_{C^{1,\al/2}(\overline{B_R})}\le C(n,R)\|\wuuj\|_{C^{1,\al/2}(\overline{B_R})}\le C(n,\al,M,R),
$$
we see that over a subsequence $$
\h_j\to \h_0 \quad\text{in }C^1(\overline{B_R}) ,
$$
for some harmonic function $\h_0\in C^1(\overline{B_R})$. Taking $j\to\infty$ in \eqref{eq:alm-min-prop}, we get $$
\int_{B_R}|\D\wuuz|^2\le \int_{B_R}|\D \h_0|^2,
$$
which implies that $\wuuz$ is the energy minimizer of the Dirichlet integral, or the harmonic function.
This finishes the proof of the claim.\\

\medskip\noindent Now, we observe that the harmonic function $\wuuz$ satisfies \begin{align}\label{eq:sup-u_0}
    \sup_{B_1}(|\wuuz|+|\D\wuuz|)=1 \quad\text{and }\sup_{B_R}\left(\frac{|\wuuz|}{R^{\ka}}+\frac{|\D\wuuz|}{R^{\ka-1}}\right)\le 1\quad\text{for any }R\ge 1.
\end{align}
On the other hand, from $x_0\in\Gamma^\ka(\uu)$, we have $\wuuj(x)=\frac{\uu(r_jx+x_0)}{jr_j^\ka}=O(|x|^{\xi})$ for some $\lfloor\ka\rfloor<\xi<\ka$. Applying Lemma~\ref{lem:alm-min-decay} yields $|\wuuj(x)|\le C_{\xi}|x|^{\xi}$, $|x|<r_{\xi}$, with $C_{\xi}$ and $r_{\xi}$ depending only on $n$, $\al$, $M$, $\xi$. This readily implies $$
|\wuuz(0)|=|\D\wuuz(0)|=\cdots=|D^{\lfloor\ka\rfloor}\wuuz(0)|=0,
$$
which combined with \eqref{eq:sup-u_0} contradicts Liouville's theorem, and \eqref{eq:growth-est-ptwise} is proved.\\
\medskip

\noindent The pointwise estimate \eqref{eq:growth-est-ptwise} tells us $\uu(x)=O(|x-x_0|^\ka)$ at every free boundary point $x_0\in\Gamma^\ka(\uu)\cap B_{1/2}$. This in turn implies, using Lemma~\ref{lem:alm-min-decay} again, the desired uniform estimate in Lemma~\ref{lem:growth-est-ka-noninteger}.
\end{proof}

In the rest of this section we establish the optimal growth of almost minimizers at free boundary points when $\ka$ is an integer. We start with weak growth estimates.

\begin{lemma}\label{lem:weak-growth-est-ka-integer}
Let $\uu\in W^{1,2}(B_1)$ be an almost minimizer in $B_1$ and $\ka\in\N$, $\ka> 2$. Then for any $\ka-1<\mu<\ka$, there exist $C>0$ and $r_0>0$, depending on $n$, $\al$, $M$, $\mu$, $E(\uu,1)$, such that $$
\sup_{B_r(x_0)}\left(\frac{|\uu|}{r^{\mu}}+\frac{|\D\uu|}{r^{\mu-1}}\right)\le C
$$
for any $x_0\in\Gamma^\ka(\uu)\cap B_{1/2}$ and $0<r<r_0$.
\end{lemma}

\begin{proof}
The proof is similar to that of Lemma~\ref{lem:growth-est-ka-noninteger}. We first claim that for each $x_0\in\Gamma^\ka(\uu)\cap B_{1/2}$, \begin{align}
    \label{eq:weak-growth-est-ptwise}
    \sup_{B_r(x_0)}\left(\frac{|\uu|}{r^\mu}+\frac{|\D\uu|}{r^{\mu-1}}\right)\le C_{\mu,x_0},\quad 0<r<r_{\mu,x_0},
\end{align}
for positive constants $C_{\mu,x_0}$ and $r_{\mu,x_0}$, depending only on $n$, $\al$, $M$, $\mu$, $E(\uu,1)$ and $x_0$. To prove it by contradiction we assume there is a sequence of positive radii $\{r_j\}^\infty_{j=1}\subset (0,1)$, $r_j\searrow0$ such that $$
\sup_{B_{r_j}(x_0)}\left(\frac{|\uu|}{r_j^\mu}+\frac{|\D\uu|}{r_j^{\mu-1}}\right)=j,\quad\sup_{B_r(x_0)}\left(\frac{|\uu|}{r^\mu}+\frac{|\D\uu|}{r^{\mu-1}}\right)\le j \quad\text{for any } r_j\le r\le 1/4.
$$
Let $$
\wuuj(x):=\frac{\uu(r_jx+x_0)}{jr_j^\mu},\quad x\in B_{\frac1{4r_j}}.
$$
Following the argument in Lemma~\ref{lem:growth-est-ka-noninteger}, we can obtain that $\wuuj\to\tilde\uu_0$ in $C^1_{\loc}(\R^n)$ for some harmonic function $\tilde\uu_0\in C^1(\R^n)$ and that $\tilde\uu_0$ satisfies \begin{align}\label{eq:sup-u_0-bound}
    \sup_{B_1}(|\wuuz|+|\D\wuuz|)=1 \quad\text{and }\sup_{B_R}\left(\frac{|\wuuz|}{R^{\mu}}+\frac{|\D\wuuz|}{R^{\mu-1}}\right)\le 1\quad\text{for any }R\ge 1.
\end{align}
On the other hand, from $x_0\in\Gamma^\ka(\uu)$, we have $\wuuj(x)=O\left(|x|^\xi\right)$ for some $\xi>\ka-1$. Thus $|\wuuj(x)|\le C_\xi|x|^\xi$ for $x\in B_{r_\xi}$ by Lemma~\ref{lem:alm-min-decay} with $C_\xi$ and $r_\xi$ depending only on $n$, $\al$, $M$, $\xi$. This readily implies $|\wuuz(0)|=|\D\wuuz(0)|=\cdots=|D^{\ka-1}\wuuz(0)|=0$, which combined with \eqref{eq:sup-u_0-bound} contradicts Liouville's theorem.

Now, the pointwise estimate \eqref{eq:weak-growth-est-ptwise} gives $\uu(x)=O(|x-x_0|^\mu)$ at every $x_0\in\Gamma^\ka(\uu)\cap B_{1/2}$, and we can apply Lemma~\ref{lem:alm-min-decay} again to conclude Lemma~\ref{lem:weak-growth-est-ka-integer}.
\end{proof}

For $0<s<1$ small to be chosen later, we define the homogeneous rescaling of $\uu$
$$
\uu_s(x):=\frac{\uu(sx)}{s^\ka},\quad x\in B_1.
$$
Recall that $\uu_s$ is an almost minimizer with gauge function $\omega(r)\le M(sr)^\al$. By Lemma~\ref{lem:weak-growth-est-ka-integer}, we have for all $\ka-1<\mu<\ka$ $$
\fint_{B_1}|\D\uu_s|^2+\fint_{B_1}|\uu_s|^2\le L_s^2,\quad L_s:=C_\mu s^{\mu-\ka}
$$
with $C_\mu$ depending on $E(\uu,1)$, $\mu$, $n$, $\al$, $M$.

\begin{lemma}\label{lem:v-p-opt-growth-integer}
Let $\uu$ and $\ka$ be as in Lemma~\ref{lem:weak-growth-est-ka-integer}, and  $0\in\Gamma^\ka(\uu)\cap B_{1/2}$. Assume that in a ball $B_r$, $r\le r_0$ universal, we have for universal constants $0<s<1$, $C_0>1$ and $\ka-1<\mu<\ka$ \begin{align}\label{eq:induc-hypo-1}
\begin{split}
    &\fint_{B_r}|\uu_s-\pp^r|^2\le L_s^2r^{2\ka},\\
    &\fint_{B_r}|\D(\uu_s-\pp^r)|^2\le L_s^2r^{2\ka-2},
\end{split}
\end{align}
with $\pp^r$ a harmonic polynomial of degree $\ka$ such that \begin{align}
    \label{eq:induc-hypo-2}
    \|\pp^r\|_{L^\infty(B_r)}\le C_0L_s^{\frac2{1+q}}r^\ka.
\end{align}
Then, there exists $\rho>0$ small universal such that \eqref{eq:induc-hypo-1} and \eqref{eq:induc-hypo-2} hold in $B_{\rho r}$ for a harmonic polynomial $\pp^{\rho r}$ of degree $\ka$.
\end{lemma}

\begin{remark}\label{rmk:opt-growth-integer}
Lemma~\ref{lem:v-p-opt-growth-integer} readily implies Theorem~\ref{thm:opt-growth} when $\ka\in\mathbb N$. In fact, as in the standard Campanato Type estimates, the lemma ensures that \eqref{eq:induc-hypo-1}-\eqref{eq:induc-hypo-2} are true for small $r\le r_1$. Combining these two estimates yields $$
\fint_{B_r}|\uu_s|^2\le CL_s^{\frac4{1+q}}r^{2\ka},\quad\fint_{B_r}|\D\uu_s|^2\le CL_s^{\frac4{1+q}}r^{2\ka}.
$$
Scaling back to $\uu(x)=s^\ka\uu_s(x/s)$ gives its optimal growth estimates at $0\in \Gamma^\ka(\uu)\cap B_{1/2}$. This also holds for any $x_0\in \Gamma^\ka(\uu)\cap B_{1/2}$ by considering $\uu(\cdot-x_0)$.
\end{remark}

\begin{proof}
For notational simplicity, we write $\vv:=\uu_s$.

\medskip\noindent\emph{Step 1.} For $0<r<1$, we denote by $\tilde\vv$ and $\tilde\pp^r$ the rescalings of $\vv$ and $\pp^r$, respectively, to the ball of radius $r$, that is $$
\tilde\vv(x):=\frac{\vv(rx)}{r^\ka} = \frac{\uu(srx)}{(s r)^\ka} ,\qquad \tilde\pp^r(x):=\frac{\pp^r(rx)}{r^\ka},\quad x\in B_1.
$$
When not specified, $\|\cdot\|_\infty$ denotes the $L^\infty$ norm in the unit ball $B_1$. With these notations, \eqref{eq:induc-hypo-1}-\eqref{eq:induc-hypo-2} read \begin{align}
    \label{eq:induc-hypo-rescal-1}
    \begin{split}
        &\fint_{B_1}|\tilde\vv-\tilde\pp^r|^2\le L_s^2,\\
        &\fint_{B_1}|\D(\tilde\vv-\tilde\pp^r)|^2\le L_s^2
    \end{split}
\end{align}
and \begin{align}
    \label{eq:induc-hypo-rescal-2}
    \|\tilde\pp^r\|_\infty\le C_0L_s^{\frac2{1+q}}.
\end{align}

We claim that if $\tilde\pp^r$ is $\ka$-homogeneous, then the finer bound $$
\|\tilde\pp^r\|_\infty\le \frac{C_0}2L_s^{\frac2{1+q}}
$$
holds for a universal constant $C_0>0$.
Indeed, applying Theorem~\ref{thm:Weiss}, Weiss-type monotonicity formula, gives $$
\int_{B_1}\left(|\D\tilde\vv|^2+2F(0,\tilde\vv)\right)-\ka(1-b(rs)^\al)\int_{\p B_1}|\tilde\vv|^2\le e^{-a(rs)^\al}W(\uu,0,0,t_0)
$$
and since $b>0$, $$
\frac{2\la_0}{1+q}\int_{B_1}|\tilde\vv|^{q+1}\le C-\left(\int_{B_1}|\D\tilde\vv|^2-\ka\int_{\p B_1}|\tilde\vv|^2\right).
$$
Using that $\tilde\pp^r$ is a $\ka$-homogeneous harmonic polynomial satisfying \eqref{eq:induc-hypo-rescal-1}, we get \begin{align*}
    \frac{2\la_0}{1+q}\int_{B_1}|\tilde\vv|^{q+1}&\le C-\left(\int_{B_1}|\D(\tilde\vv-\tilde\pp^r)|^2-\ka\int_{\p B_1}|\tilde\vv-\tilde\pp^r|^2\right)\\
    &\le C(1+L_s^2).
\end{align*}
Thus \begin{align*}
    \int_{B_1}|\tilde\pp^r|^{q+1}&\le C\int_{B_1}\left(|\tilde\vv|^{q+1}+|\tilde\vv-\tilde\pp^r|^{q+1}\right)\\
    &\le C\left(1+L_s^2+\|\tilde\vv-\tilde\pp^r\|_{L^2(B_1)}^{q+1}\right)
    \le C(1+L_s^2+L_s^{q+1}).
\end{align*}
In conclusion (for a universal constant $C>0$) we have   $
\|\tilde\pp^r\|_\infty^{q+1}\le CL_s^2$ from which we deduce that $$
\|\tilde\pp^r\|_\infty\le \frac{C_0}2L_s^{\frac{2}{1+q}}.
$$

\medskip\noindent\emph{Step 2.} We claim that for some $t_0>0$ small universal \begin{align}
    \label{eq:v-rescal-est}
    |\D\tilde\vv(x)|\le C_\mu L_s^{\frac2{1+q}}|x|^{\mu-1},\,\, |\tilde\vv(x)|\le C_\mu L_s^{\frac2{1+q}}|x|^\mu,\quad x\in B_{t_0}.
\end{align}
Indeed, \eqref{eq:induc-hypo-rescal-1}-\eqref{eq:induc-hypo-rescal-2} give ($C$ universal possibly changing from equation to equation) \begin{align*}
    \fint_{B_1}|\D\tilde\vv|^2\le 2\fint_{B_1}\left(|\D(\tilde\vv-\tilde\pp^r)|^2+|\D\tilde\pp^r|^2\right)\le C(L_s^2+L_s^{\frac4{1+q}})\le CL_s^{\frac4{1+q}}.
\end{align*} Similarly, $$
\fint_{B_1}|\tilde\vv|^2\le CL_s^{\frac4{1+q}},
$$
and H\"oder's inequality gives $$
\fint_{B_1}|\tilde\vv|^{1+q}\le CL_s^2.
$$
We conclude that the following  energy estimate
$E( \tilde\ww,1)\le C,
$
where  $\tilde\ww:=L_s^{-\frac2{1+q}}\tilde\vv$ is an almost minimizer with the same gauge function as $\tilde\vv$ for the energy $$
\int_{B_t(x_0)}\left(|\D\tilde\ww|^2+2L_s^{\frac{2(q-1)}{q+1}}F(x_0,\tilde\ww)\right),\quad0<t<1.
$$
As before, $L_s^{\frac{2(q-1)}{q+1}}\le1$ allows us to repeat the arguments towards the $C^{1,\al/2}$-estimtate of almost minimizers as well as towards Lemma~\ref{lem:weak-growth-est-ka-integer}. Since $\uu=o(|x|^{\ka-1})$ implies $\tilde\ww=o(|x|^{\ka-1})$, we can apply Lemma~\ref{lem:weak-growth-est-ka-integer} to have $$
|\tilde\ww(x)|\le C_\mu|x|^\mu,\,\, |\D\tilde\ww(x)|\le C_\mu|x|^{\mu-1},\quad x\in B_{t_0}.
$$
This readily implies \eqref{eq:v-rescal-est}.

\medskip\noindent\emph{Step 3.} Let $\tilde\h$ be the harmonic replacement of $\tilde\vv$ in $B_{t_0}$. Then, we claim that \begin{align}
    \label{eq:v-h-rescal-diff-est}
    \fint_{B_{t_0}}|\D(\tilde\vv-\tilde\h)|^2\le CL_s^{1+\frac{2q}{1+q}}.
\end{align}
Let us first recall that $\vv(x)=\frac{\uu(sx)}{s^\ka}$ and that $\tilde\vv(x)=\frac{\vv(rx)}{r^\ka}=\frac{\uu(rsx)}{(rs)^\ka}$ is an almost minimizer with gauge function $\omega(\rho)\le M(rs\rho)^\al$. Thus,
 \begin{align*}
    \fint_{B_{t_0}}|\D(\tilde\vv-\tilde\h)|^2&\le M(rs)^\al\fint_{B_{t_0}}(|\D\tilde\h|^2+2F(\tilde\h))+2\fint_{B_{t_0}}(F(0,\tilde\h)-F(0,\tilde\vv))\\
    &=:I+II.
\end{align*}
To estimate $I$, we use that $\tilde\h$ is the harmonic replacement of $\tilde\vv$, together with \eqref{eq:v-rescal-est}, to get $$
\fint_{B_{t_0}}|\D\tilde\h|^2\le\fint_{B_{t_0}}|\D\tilde\vv|^2\le CL_s^{\frac4{1+q}}.
$$
In addition, it follows from the subharmonicity of $|\tilde\h|^2$  and \eqref{eq:v-rescal-est} that $$
\|\tilde\h\|_{L^\infty(B_{t_0})}=\|\tilde\h\|_{L^\infty(\p B_{t_0})}=\|\tilde\vv\|_{L^\infty(\p B_{t_0})}\le CL_s^{\frac2{1+q}}.
$$
This gives $$
\fint_{B_{t_0}}F(0,\tilde\h)\le C\|\tilde\h\|_{L^\infty(B_{t_0})}^{1+q}\le CL_s^2\le CL_s^{\frac4{1+q}}.
$$
Therefore, $$
I\le Cs^\al L_s^{\frac4{1+q}}=CL_s^{\frac\al{\mu-\ka}+\frac4{1+q}}\le CL_s^{1+\frac{2q}{1+q}},
$$
where the last inequality holds if $\mu$ is chosen universal close enough to $\ka$ (specifically, $\mu\ge\ka-\frac{\al(1+q)}{3(1-q)})$.

Next, we estimate $II$. \begin{align*}
    \fint_{B_{t_0}}|F(0,\tilde\h)-F(0,\tilde\vv)|&\le C\fint_{B_{t_0}}\left||\tilde\h|^{1+q}-|\tilde\vv|^{1+q}\right|\\
    &\le C\fint_{B_{t_0}}\left(|\tilde\h|^q+|\tilde\vv|^q\right)|\tilde\vv-\tilde\h|\\
    &\le C\left(\|\tilde\h\|_{L^\infty(B_{t_0})}^q+\|\tilde\vv\|_{L^\infty(B_{t_0})}^q\right)\left(\fint_{B_{t_0}}|\tilde\vv-\tilde\h|^2\right)^{1/2}\\
    &\le CL_s^{\frac{2q}{1+q}}\left(\fint_{B_{t_0}}|\D(\tilde\vv-\tilde\h)|^2\right)^{1/2}.
\end{align*}
To bound the last term, we observe \begin{align*}
    \int_{B_{t_0}}\D(\tilde\h-\tilde\pp^r)\cdot\D(\tilde\h-\tilde\vv)&=\int_{\p B_{t_0}}\p_\nu(\tilde\h-\tilde\pp^r)(\tilde\h-\tilde\vv)-\int_{B_{t_0}}\Delta(\tilde\h-\tilde\pp^r)(\tilde\h-\tilde\vv)\\
    &=0,
\end{align*}
and use it to obtain \begin{align*}
    \int_{B_{t_0}}|\D(\tilde\vv-\tilde\pp^r)|^2-\int_{B_{t_0}}|\D(\tilde\vv-\tilde\h)|^2&=\int_{B_{t_0}}|\D\tilde\pp^r|^2-|\D\tilde\h|^2-2\D\tilde\vv\cdot\D(\tilde\pp^r-\tilde\h)\\
    &=\int_{B_{t_0}}\D(\tilde\pp^r-\tilde\h)\cdot\D(\tilde\pp^r+\tilde\h-2\tilde\vv)\\
    &=\int_{B_{t_0}}\D(\tilde\pp^r-\tilde\h)\cdot\D(\tilde\pp^r-\tilde\h)\\
    &=\int_{B_{t_0}}|\D(\tilde\pp^r-\tilde\h)|^2\ge 0.
\end{align*}
Therefore, \begin{align*}
    II&\le 2\fint_{B_{t_0}}|F(0,\tilde\h)-F(0,\tilde\vv)|\le CL_s^{\frac{2q}{1+q}}\left(\fint_{B_{t_0}}|\D(\tilde\vv-\tilde\h)|^2\right)^{1/2}\\
    &\le CL_s^{\frac{2q}{1+q}}\left(\fint_{B_{t_0}}|\D(\tilde\vv-\tilde\pp^r)|^2\right)^{1/2}\le CL_s^{1+\frac{2q}{1+q}},
\end{align*}
where we used \eqref{eq:induc-hypo-rescal-1} in the last inequality. This completes the proof of \eqref{eq:v-h-rescal-diff-est}.

\medskip\noindent\emph{Step 4.} For $\rho\in(0,t_0)$ small to be chosen below, we have by \eqref{eq:v-h-rescal-diff-est} $$
\fint_{B_\rho}|\D(\tilde\vv-\tilde\h)|^2\le C\rho^{-n}L_s^{1+\frac{2q}{1+q}}.
$$
Since $\tilde\h-\tilde\pp^r$ is harmonic, arguing as in the proof of Lemma~\ref{lem:alm-min-decay}, we can find a harmonic polynomial $\qq^r$ (in $B_r$) of degree $\kappa$  such that $\tilde \qq^r (x)=\frac{\qq^r (rx)}{r^\kappa} $ satisfies
$$
\fint_{B_\rho}|\D(\tilde\h-\tilde\pp^r-\tilde\qq^r)|^2\le C\rho^{2\ka}\fint_{B_{t_0}}|\D(\tilde\h-\tilde\pp^r)|^2.
$$
Using \eqref{eq:induc-hypo-rescal-1} and \eqref{eq:v-h-rescal-diff-est}, we further have \begin{align*}
    \fint_{B_\rho}|\D(\tilde\h-\tilde\pp^r-\tilde\qq^r)|^2&\le C\rho^{2\ka}\left(\fint_{B_{t_0}}|\D(\tilde\h-\tilde\vv)|^2+\fint_{B_{t_0}}|\D(\tilde\vv-\tilde\pp^r)|^2\right)\\
    &\le C\rho^{2\ka}(L_s^{1+\frac{2q}{1+q}}+L_s^2)\\
    &\le CL_s^2\rho^{2\ka}.
\end{align*}
This, combined with the equation above, gives
\begin{align*}
    \fint_{B_\rho}|\D(\tilde\vv-\tilde\pp^r-\tilde\qq^r)|^2&\le 2\fint_{B_\rho}(|\D(\tilde\h-\tilde\pp^r-\tilde\qq^r)|^2+|\D(\tilde\vv-\tilde\h)|^2)\\
    &\le L_s^2\rho^{2\ka-2}(C\rho^2+C\rho^{-n-2\ka+2}L_s^{\frac{q-1}{1+q}}).
\end{align*}
By possibly modifying $\tilde\qq^r$ by adding a constant, we also have by Poincar\'e inequality \begin{align*}
    \fint_{B_\rho}|\tilde\vv-\tilde\pp^r-\tilde\qq^r|^2&\le C\rho^2\fint_{B_\rho}|\D(\tilde\vv-\tilde\pp^r-\tilde\qq^r)|^2\\
    &\le L_s^2\rho^{2\ka}(C_1\rho^2+C_1\rho^{-n-2\ka+2}L_s^{\frac{q-1}{1+q}}).
\end{align*}
One can see that $\tilde\qq^r$ depends on $\rho$ as well as $r$, but we keep denoting $\tilde\qq^r$ for the notational simplicity. We choose $\rho\in(0,t_0)$ small so that $$
C_1\rho^2\le 1/8
$$
and then choose $L_s$ large (that is $s$ small) so that $$
C_1\rho^{-n-2\ka+2}L_s^{\frac{q-1}{1+q}}\le 1/8.
$$
This yields that \begin{align}
    \label{eq:u-p-q-rescal-est}
    \begin{split}
        &\fint_{B_\rho}|\tilde\vv-\tilde\pp^r-\tilde\qq^r|^2\le \frac14L_s^2\rho^{2\ka},\\
        &\fint_{B_\rho}|\D(\tilde\vv-\tilde\pp^r-\tilde\qq^r)|^2\le \frac14L_s^2\rho^{2\ka-2}.
    \end{split}
\end{align}
Notice that \eqref{eq:u-p-q-rescal-est} holds for any $\rho\in[\rho_1,\rho_2]$, with some constants $\rho_1$, $\rho_2>0$ small universal and $L_s>0$ large universal.

In addition, we have \begin{align}\label{eq:q-rescal-est}\begin{split}
    \|\tilde\qq^r\|_{L^\infty(B_1)}&\le C(\rho_2,\ka,n)\|\tilde\qq^r\|_{L^\infty(B_{\rho_2/2})}\le C\left(\fint_{B_{\rho_2}}|\tilde\qq^r|^2\right)^{1/2}\\
    &\le C\left(\fint_{B_{\rho_2}}|\tilde\vv-\tilde\pp^r-\tilde\qq^r|^2+\fint_{B_{\rho_2}}|\tilde\vv-\tilde\pp^r|^2\right)^{1/2}\\
    &\le \bar{C}L_s,
\end{split}\end{align}
where the last line follows from \eqref{eq:induc-hypo-rescal-1} and \eqref{eq:u-p-q-rescal-est}. We remark that $\bar C$ depends on $\rho_2$, but is independent of $\rho_1$.

\medskip\noindent\emph{Step 5.} In this step, we prove that the estimates \eqref{eq:induc-hypo-1}-\eqref{eq:induc-hypo-2} over $B_r$ imply the same estimates   over $B_{\rho r}$. 
We set (by abuse of notation)
$\pp^{\rho r}:=\pp^r+\qq^r$ and recall $\qq^r(x):=r^\ka\tilde\qq^r\left(\frac{x}{r}\right)$.
Following the notations above we denote its homogeneous rescaling by $$
\tilde\pp^{\rho r}(x):=\frac{\pp^{\rho r}(\rho rx)}{(\rho r)^\ka}=\frac{(\tilde\pp^r+\tilde\qq^r)(\rho x)}{\rho^\ka}.
$$
We divide the proof into the following two cases: $$
\text{either }\pp^{\rho r}\text{ is homogeneous of degree }\ka\text{ or not}.
$$

\noindent\emph{Case 1.} Suppose that $\pp^{\rho r}$ is $\ka$-homogeneous. Then \eqref{eq:induc-hypo-1} over $B_{\rho r}$ follows from \eqref{eq:u-p-q-rescal-est} and \eqref{eq:induc-hypo-2} over  $B_{\rho r}$ with $\pp^{\rho r}$ from the monotonicity formula, see the claim in \emph{Step 1}.

\medskip\noindent\emph{Case 2.} Now we assume that $\pp^{\rho r}$ is not homogeneous of degree $\ka$. Note that for each polynomial $\pp$ of degree $\ka$, we can decompose $\pp=\pp_h+\pp_i$
 with $\pp_h$, $\pp_i$ respectively the $\ka$-homogeneous and the inhomogeneous parts of $\pp$. We will prove that \eqref{eq:induc-hypo-1}-\eqref{eq:induc-hypo-2} hold in $B_{\rho r}$ with the harmonic polynomial $\pp^{\rho r}_h$ (in the place of $\pp^{\rho r}$). In fact, it is enough to prove the statement in $B_{\rho r}$ under the assumption that $\pp^r$ is $\ka$-homogeneous. Indeed, we note that \eqref{eq:u-p-q-rescal-est} holds for every any $r\le1$ and $\rho_1<\rho<\rho_2$. Thus, if $r\le r_0$ is small enough, we can find $\rho\in[\rho_1,\rho_2]$ such that $r=\rho^m$ for some $m\in\mathbb N$. Then, we can iterate the above statement with such $\rho$, starting with $\pp^1=0$.

Now, we distinguish two subcases, for $\de>0$ small and $L_s>0$ large universal: $$
\|\tilde\qq^r_i\|_\infty\le \de L_s\quad\text{or}\quad\|\tilde\qq^r_i\|_\infty>\de L_s.
$$

\noindent\emph{Case 2.1.} We first consider the case $\|\tilde\qq^r_i\|_\infty\le \de L_s$. To prove \eqref{eq:induc-hypo-1}, we use the $\ka$-homogeneity of $\pp^r$ to have $$
\pp^{\rho r}_i=\pp^r_i+\qq^r_i=\qq^r_i,
$$
which implies (in accordance with the decomposition above) 
$$
\pp^{\rho r}_h=\pp^{\rho r}-\pp^{\rho r}_i=\pp^{\rho r}-\qq^r_i.
$$
Combining this with \eqref{eq:u-p-q-rescal-est} gives
\begin{align*}
    \fint_{B_{\rho r}}|\vv-\pp^{\rho r}_h|^2&\le 2\fint_{B_{\rho r}}|\vv-\pp^{\rho r}|^2+2\fint_{B_{\rho r}}|\qq^r_i|^2\le 2\left(\frac{L_s^2}4(\rho r)^{2\ka}\right)+2r^{2\ka}\fint_{B_\rho}|\tilde\qq^r_i|^2\\
    &\le \frac{L_s^2}2(\rho r)^{2\ka}+2L_s^2\de^2r^{2\ka}\le L_s^2(\rho r)^{2\ka},\quad\text{if }\de\le\frac{\rho_1^\ka}2.
\end{align*}
Similarly, \begin{align*}
    &\fint_{B_{\rho r}}|\D(\vv-\pp^{\rho r}_h)|^2\\
    &\qquad\le 2\fint_{B_{\rho r}}|\D(\vv-\pp^{\rho r})|^2+2\fint_{B_{\rho r}}|\D\qq^r_i|^2\le \frac{L_s^2}2(\rho r)^{2\ka-2}+2r^{2\ka-2}\fint_{B_\rho}|\D\tilde\qq^r_i|^2\\
    &\qquad\le \frac{L_s^2}2(\rho r)^{2\ka-2}+C_2r^{2\ka-2}\|\tilde\qq^r_i||_{L^\infty(B_1)}^2\le \frac{L_s^2}2(\rho r)^{2\ka-2}+C_2\de^2L_s^2r^{2\ka-2}\\
    &\qquad\le L_s^2(\rho r)^{2\ka-2},\quad\text{if }\de\le\frac{\rho_1^{\ka-1}}{2C_2^{1/2}}.
\end{align*}
This proves \eqref{eq:induc-hypo-1} in $B_{\rho r}$ with harmonic polynomial $\pp_h^{\rho r}$. \eqref{eq:induc-hypo-2} follows from the homogeneity of $\pp^{\rho r}_h$.

\medskip\noindent\emph{Case 2.2.} Now we assume $\|\tilde\qq^r_i\|_\infty>\de L_s$. We will show that it leads to a contradiction and that we always fall in the previous case.

Indeed, for $\bar r:=\rho_1 r$, \begin{align*}
    \|\tilde\pp_i^{\bar r}\|_\infty=\left\|\frac{\tilde\qq^r_i(\rho_1x)}{\rho_1^\ka}\right\|_\infty=\frac{\|\tilde\qq^r_i\|_{L^\infty(B_{\rho_1})}}{\rho_1^\ka}\ge\frac{\|\tilde\qq^r_i\|_{L^\infty(B_1)}}{\rho_1}\ge\frac\de{\rho_1}L_s.
\end{align*}
Recall that the constant $\bar C$ in \eqref{eq:q-rescal-est} is independent of $\rho_1$. Thus, for $\rho_1>0$ small, $$
\|\tilde\pp^{\bar r}_i\|_\infty\ge C_3L_s,\quad C_3\ge 2\bar C.
$$
Similarly, $$
\|\tilde\pp^{\bar r}_i\|_\infty=\frac{\|\tilde\qq^r_i\|_{L^\infty(B_{\rho_1})}}{\rho_1^\ka}\le \frac{\|\tilde\qq^r\|_{L^\infty(B_{1})}}{\rho_1^\ka}\le C_4(\rho_1)L_s.
$$
For $L_s$ large \begin{align*}
    \|\tilde\pp^{\bar r}_h\|_\infty=\|\tilde\pp^r+\tilde\qq^r_h\|_\infty\le C_0L_s^{\frac2{1+q}}+\bar CL_s\le \frac34C_0L_s^{\frac2{1+q}},
\end{align*}
and thus \begin{align*}
    \|\tilde\pp^{\bar r}\|_\infty\le \|\tilde\pp^{\bar r}_h\|_\infty+\|\tilde\pp^{\bar r}_i\|_\infty\le \frac34C_0L_s^{\frac2{1+q}}+C_4L_s\le \frac78C_0L_s^{\frac2{1+q}}.
\end{align*}
We iterate again, and conclude that $$
\frac{\rho_1^{-1}}2\|\tilde\pp^{\bar r}_i\|_\infty\le\|\tilde\pp^{\rho_1\bar r}_i\|_\infty\le 2\rho_1^{-\ka}\|\tilde\pp^{\bar r}_i\|_\infty
$$
while $$
\|\tilde\pp^{\rho_1\bar r}_h-\tilde\pp^{\bar r}_h\|_\infty\le \bar CL_s.
$$
Indeed, using that $\|\tilde\pp^{\bar r}_i\|_\infty\ge C_3L_s\ge 2\|\tilde\qq^{\bar r}\|_\infty$, we get \begin{align*}
    \|\tilde\pp^{\rho_1\bar r}_i\|_\infty&=\frac{\|\tilde\pp^{\bar r}_i+\tilde\qq^{\bar r}_i\|_{L^\infty(B_{\rho_1})}}{\rho_1^\ka}\ge\frac{\|\tilde\pp^{\bar r}_i+\tilde\qq^{\bar r}_i\|_{L^\infty(B_1)}}{\rho_1}\ge \frac{\|\tilde\pp^{\bar r}_i\|_\infty-\|\tilde\qq^{\bar r}\|_\infty}{\rho_1}\\
    &\ge\frac{\rho_1^{-1}}2\|\tilde\pp^{\bar r}_i\|_\infty,
\end{align*}
and similarly $$
 \|\tilde\pp^{\rho_1\bar r}_i\|_\infty\le\frac{\|\tilde\pp^{\bar r}_i\|_\infty+\|\tilde\qq^{\bar r}\|_\infty}{\rho_1^\ka}\le 2\rho_1^{-\ka}\|\tilde\pp^{\bar r}_i\|_\infty.
$$
In addition, $$
\|\tilde\pp^{\rho_1\bar r}_h-\tilde\pp^{\bar r}_h\|_\infty=\|\tilde\qq^{\bar r}_h\|_\infty\le \bar CL_s.
$$
In conclusion, if $a_l:=\|\tilde\pp^{\rho_1^l\bar r}_i\|_\infty$ and $b_l:=\|\tilde\pp^{\rho_1^l\bar r}_h\|_\infty$, $l\ge0$, we can iterate and have that ($\rho_1$ small) \begin{align*}
    &C_3L_s\le a_0\le C_4L_s,\quad 2a_l\le a_{l+1}\le C^*(\rho_1)a_l,\\
    &b_0\le\frac34C_0L_s^{\frac2{1+q}},\quad |b_{l+1}-b_l|\le \bar CL_s,
\end{align*}
as long as $$
a_l+b_l\le C_0L_s^{\frac2{1+q}},
$$
which holds at $l=0$. This iteration is possible, since we can repeat \emph{Step 1--4}  as long as $\|\tilde\pp^{\rho_1^l\bar r}\|_\infty\le a_l+b_l\le C_0L_s^{\frac2{1+q}}$.

Thus we can iterate till the first $l\ge 1$ ($l \sim \log L_s)$ such that for $c_0\le\frac1{2C^*}$ small universal \begin{align}
    \label{eq:a_l-bound}\frac{C_0}8L_s^{\frac2{1+q}}\ge a_l\ge c_0L_s^{\frac2{1+q}},\\
    \label{eq:b_l-bound}b_l\le\frac78C_0L_s^{\frac2{1+q}}.
\end{align}
We will now finally show that these inequalities lead to a contradiction.

For simplicity we write $\tilde\vv^l:=\tilde\vv^{\rho_1^l\bar r}$ and $\tilde\pp^l:=\tilde\pp^{\rho_1^l\bar r}$. From $$
\fint_{B_1}|\tilde\vv^l-\tilde\pp^l|^2\le L_s^2,
$$
we see that for any $\eta \leq 1$, 
$$L_s^{-\frac{2}{1+q}} \left(\fint_{B_\eta} |\tilde \vv^l - \tilde \pp^l|^2\right)^{1/2} \to 0, \quad L_s \to \infty.$$
Moreover, from \eqref{eq:v-rescal-est} and \eqref{eq:b_l-bound}, for $\mu> \kappa-1,$ and $|x| \leq 1/2,$
$$L_s^{-\frac{2}{1+q}}|\tilde\vv^l - \tilde \pp_h^l|(x) \leq C |x|^\mu + \frac78C_0 |x|^\kappa.$$ This combined with the equation above gives that 
$$L_s^{-\frac{2}{1+q}} \left(\fint_{B_\eta}  |\tilde \pp_i^l|^{2}\right)^{1/2} \leq C(\eta^\mu + \eta^\kappa) + o(1/L_s).$$
Hence, we conclude from \eqref{eq:a_l-bound} that 
$$c_0 \eta^{\kappa-1} \leq L_s^{-\frac{2}{1+q}} \left(\fint_{B_\eta} |\tilde \pp_i^l|^{2}\right)^{1/2} \leq C(\eta^\mu + \eta^\kappa) + o(1/L_s).$$
Since $\mu> \kappa-1$, for $\eta$ small and $L_s$ large we obtain a contradiction.
\end{proof}

Before closing this section, we notice that the combination of Theorem~\ref{thm:grad-u-holder} and Theorem~\ref{thm:opt-growth} provides the following optimal regularity of an almost minimizer at the free boundary.

\begin{corollary}\label{UB} Let $\uu$ be an almost minimizer in $B_1$. Then, for $x_0\in \Gamma^\ka(\uu)\cap B_{1/2}$, $0<r<1/2$, $\uu_{x_0,r}\in C^{1,\al/2}(B_1)$ and for any $K\Subset B_1$, \begin{equation}\label{1a}
\|\uu_{x_0,r}\|_{C^{1,\al/2}(K)}\le C(n,\al,M, \ka, K,E(\uu,1)).
\end{equation}
\end{corollary}


\section{Non-degeneracy}\label{sec:nondeg}
In this section we shall derive an important non-degeneracy property of almost minimizers, Theorem~\ref{thm:nondeg}. 

In the rest of this paper, for $x_0\in B_{1/2}$ and $0<r<1/2$ we denote the $\ka$-homogeneous rescalings of $\uu$ by $$
\uu_{x_0,r}(x):=\frac{\uu(rx+x_0)}{r^\ka},\quad x\in B_{1/(2r)}.
$$

\begin{theorem}[Non-Degeneracy]\label{thm:nondeg}Let $\uu$ be an almost minimizer in $B_1$. There exist constants $c_0=c_0(q,n,\alpha,M, E(\uu, 1))>0$ and $r_0=r_0(q,n,\alpha,M)>0$ such that if $x_0\in\Gamma^\kappa(\uu)\cap B_{1/2}$ and $0<r<r_0$, then $$
\sup_{B_{r}(x_0)}|\uu|\ge c_0r^\kappa.
$$

\end{theorem}

The proof of Theorem \ref{thm:nondeg} relies on the following lemma.

\begin{lemma}\label{imp}
Let $\uu$ be an almost minimizer in $B_1$. Then, there exist small constants $\e_0=\e_0(q,n,M)>0$ and $r_0=r_0(q,n,\al,M)>0$ such that for $0<r<r_0$ and $x_0\in B_{1/2}$, if $E(\uu_{x_0,r},1)\le \e_0$ then $E(\uu_{x_0,r/2},1)\le \e_0$.
\end{lemma}

\begin{proof}
For simplicity we may assume $x_0=0$. For $0<r<r_0$ to be specified later, let $\vv_r$ be a solution of $\Delta\vv_r=f(0,\vv_r)$ in $B_1$ with $\vv_r=\uu_r$ on $\p B_1$. We claim that if $\e_0=\e_0(q,n,M)>0$ is small, then $\vv_r\equiv \mathbf{0}$ in $B_{1/2}$. Indeed, if not, then $\sup_{B_{3/4}}|\vv_r|\ge c_0(q,n)$ by the non-degeneracy of the solution $\vv_r$. Thus $|\vv_r(z_0)|\ge c_0(q,n)$ for some $z_0\in \overline{B_{3/4}}$. Moreover, from $1/M\le\la_\pm\le M$ and $0<q<1$, we have $1/M\le\frac2{1+q}\la_\pm\le 2M$, thus $$
1/M|\vv_r|^{q+1}\le 2F(0,\vv_r)\quad\text{and}\quad2F(0,\uu_r)\le 2M|\uu_r|^{q+1}.
$$
Then \begin{align*}
    E(\vv_r,1)&\le M\int_{B_1}\left(|\D\vv_r|^2+2F(0,\vv_r)\right)\le M\int_{B_1}\left(|\D\uu_r|^2+2F(0,\uu_r)\right) \\&\le 2M^2E(\uu_r,1)\le 2M^2\e_0.
\end{align*}
Combining this with the estimate for the solution $\vv_r$ gives
$$\sup_{B_{7/8}}|\D\vv_r|\le C(n,M)\left(E(\vv_r,1)^{1/2}+1\right)\le C(n,M),
$$
hence $$
|\vv_r|\ge\frac{c_0(q,n)}2\quad\text{in}\quad B_{\rho_0}(z_0)
$$
for some small $\rho_0=\rho_0(q,n,M)>0$. This implies that $$
c(q,n,M)\le\int_{B_{\rho_0}(z_0)}|\vv_r|^{q+1}\le E(\vv_r,1)\le 2M^2\e_0,
$$
which is a contradiction if $\e_0=\e_0(q,n,M)$ is small.\\

Now, we use $E(\vv_r,1)\le 2M^2\e_0$ together with the fact that $\uu_r$ satisfies \eqref{eq:alm-min-frozen} in $B_1$ with gauge function $\omega_r(\rho)\le M(r\rho)^\al$ to get \begin{align*}
    \int_{B_1}\left(|\D\uu_r|^2+2F(0,\uu_r)\right)&\le (1+ Mr^\alpha)\int_{B_1}\left(|\D\vv_r|^2+2F(0,\vv_r)\right)\\
    &\le 4M^4\e_0 r^\al+\int_{B_1}\left(|\D\vv_r|^2+2F(0,\vv_r)\right),
\end{align*}
thus \begin{align*}
    &\int_{B_1}\left(|\D\uu_r|^2-|\D\vv_r|^2\right)\\
    &\le 4M^4\e_0 r^\al+2\int_{B_1}\left(F(0,\vv_r)-F(0,\uu_r)\right)\\
    &=4M^4\e_0r^\al+\frac{2\la_+(0)}{1+q}\int_{B_1}\left(|\vv_r^+|^{q+1}-|\uu_r^+|^{q+1}\right)+\frac{2\la_-(0)}{1+q}\int_{B_1}\left(|\vv_r^-|^{q+1}-|\uu_r^-|^{q+1}\right).
 \end{align*}
This gives \begin{align}\label{eq:u_r-v_r-diff-est}\begin{split}
    &\int_{B_1}|\D(\uu_r-\vv_r)|^2
    =\int_{B_1}\left(|\D\uu_r|^2-|\D\vv_r|^2-2\D(\uu_r-\vv_r)\cdot\D\vv_r\right)\\
    &=\int_{B_1}\left(|\D\uu_r|^2-|\D\vv_r|^2\right)+2\int_{B_1}(\uu_r-\vv_r)f(0,\vv_r)\\
    &\begin{multlined}=\int_{B_1}\left(|\D\uu_r|^2-|\D\vv_r|^2\right)+2\la_+(0)\int_{B_1}|\vv_r^+|^{q-1}(\uu_r-\vv_r)\cdot\vv_r^+\\
    -2\la_-(0)\int_{B_1}|\vv_r|^{q-1}(\uu_r-\vv_r)\cdot\vv_r^-\end{multlined}\\
    &\begin{multlined}\le 4M^4\e_0r^\al+\frac{2\la_+(0)}{1+q}\int_{B_1}\left(|\vv_r^+|^{q+1}-|\uu_r^+|^{q+1}+(1+q)|\vv_r^+|^{q-1}(\uu_r-\vv_r)\cdot\vv_r^+\right)\\
    +\frac{2\la_-(0)}{1+q}\int_{B_1}\left(|\vv_r^-|^{q+1}-|\uu_r^-|^{q+1}-(1+q)|\vv_r^-|^{q-1}(\uu_r-\vv_r)\cdot\vv_r^-\right)\end{multlined}\\
    &\begin{multlined}=4M^4\e_0r^\al+\frac{2\la_+(0)}{1+q}\int_{B_1}\left((1+q)|\vv_r^+|^{q-1}\uu_r\cdot\vv_r^+-q|\vv_r^+|^{q+1}-|\uu_r^+|^{q+1}\right)\\
    +\frac{2\la_-(0)}{1+q}\int_{B_1}\left(-(1+q)|\vv_r^-|^{q-1}\uu_r\cdot\vv_r^--q|\vv_r^-|^{q+1}-|\uu_r^-|^{q+1}\right).
    \end{multlined}
\end{split}\end{align}
To compute the last two terms, we use $
\uu_r\cdot\vv_r^+\le\uu_r^+\cdot\vv_r^+\le|\uu_r^+||\vv_r^+|$ and Young's inequality to get$$
|\vv_r^+|^{q-1}\uu_r\cdot\vv_r^+\le |\uu_r^+||\vv_r^+|^q\le\frac1{q+1}|\uu_r^+|^{q+1}+\frac{q}{1+q}|\vv_r^+|^{q+1}.
$$
Similarly, from $-\uu_r\cdot\vv_r^-\le\uu_r^-\cdot\vv_r^-\le|\uu_r^-||\vv_r^-|,$ we also have $$
-|\vv_r^-|^{q-1}\uu_r\cdot\vv_r^-\le |\uu_r^-||\vv_r^-|^q\le\frac1{q+1}|\uu_r^-|^{q+1}+\frac{q}{1+q}|\vv_r^-|^{q+1}.
$$
Combining those inequalities  and \eqref{eq:u_r-v_r-diff-est} yields $$
\int_{B_1}|\D(\uu_r-\vv_r)|^2\le 4M^4\e_0r^\al.
$$
Applying Poincar\'e inequality and H\"older's inequality, we obtain $$
\int_{B_1}\left(|\D(\uu_r-\vv_r)|^2+|\uu_r-\vv_r|^{q+1}\right)\le C(q,n,M)r^{\al/2}.
$$
Since $\vv_r\equiv\mathbf0$ in $B_{1/2}$, we see that for $0<r<r_0(q,n,\al,M)$, \begin{align*}
    E(\uu_r,1/2)=\int_{B_{1/2}}\left(|\D\uu_r|^2+|\uu_r|^{q+1}\right)\le C(q,n,M)r^{\al/2}\le\frac{\e_0}{2^{n+2\ka-2}}.
\end{align*}
Therefore, we conclude that \begin{equation*}
E(\uu_{r/2},1)=2^{n+2\ka-2}E(\uu_r,1/2)\le \e_0.\qedhere
\end{equation*}
\end{proof}

Lemma \ref{imp} immediately implies the following integral form of non-degeneracy.

\begin{lemma}\label{lem:nondeg-int}
Let $\uu$, $\e_0$ and $r_0$ be as in the preceeding lemma. If $x_0\in \overline{\{|\uu|>0\}}\cap B_{1/2}$ and $0<r<r_0$, then \begin{align}
    \label{eq:alm-min-nondeg}
    \int_{B_r(x_0)}\left(|\D\uu|^2+|\uu|^{q+1}\right)\ge \e_0 r^{n+2\ka-2}.
\end{align}
\end{lemma}

\begin{proof}
By the continuity of $\uu$, it is enough to prove \eqref{eq:alm-min-nondeg} for $x_0\in\{|\uu|>0\}\cap B_{1/2}$. Towards a contradiction, we suppose that $\int_{B_r(x_0)}\left(|\D\uu|^2+|\uu|^{q+1}\right)\le\e_0r^{n+2\ka-2}$, or equivalently $E(\uu_{x_0,r},1)\le \e_0$. Then, by the previous lemma we have $E(\uu_{x_0,r/2^j},1)\le\e_0$ for all $j\in\N$. From $|\uu(x_0)|>0$, we see that $|\uu|>c_1>0$ in $B_{r/2^j}(x_0)$ for large $j$. Therefore, \begin{align*}
    \e_0&\ge E(\uu_{x_0,r/2^j},1)=\frac1{(r/2^j)^{n+2\ka-2}}\int_{B_{r/2^j}(x_0)}\left(|\D\uu|^2+|\uu|^{q+1}\right)\\
    &\ge\frac1{(r/2^j)^{n+2\ka-2}}\int_{B_{r/2^j}(x_0)}2c_1^{q+1}= \frac{C(n)c_1^{q+1}}{(r/2^j)^{2\ka-2}}\to\infty\quad\text{as }j\to\infty.
\end{align*}
This is a contradiction, as desired.
\end{proof}

We are now ready to prove Theorem~\ref{thm:nondeg}.

\begin{proof}[Proof of Theorem~\ref{thm:nondeg}] Assume by contradiction that 
$$\uu_{x_0,r}(x) < c_0, \quad   x\in B_{1}, $$ with $c_0$ small, to be made precise later. Let $\epsilon_0, r_0$ be the constants in Lemma \ref{imp} and $\omega_n=|B_1|$ be the volume of an $n$-dimensional ball. For $r<r_0$, by interpolation together with the estimate \eqref{1a},
\begin{align*}\|\nabla \uu_{x_0,r}\|_{L^\infty (B_{1/2})} & \leq \epsilon \|\uu_{x_0,r}\|_{C^{1,\alpha/2}(B_{3/4})} + K(\epsilon) \|\uu_{x_0,r}\|_{L^\infty(B_{3/4})}\\
&\le \epsilon  C(n,\alpha,M, E(\uu, 1)) + K(\epsilon)c_0 \leq \frac{\epsilon_0}{2^{n+2\ka-1}\omega_n^{1/2}},
 \end{align*}
by choosing $\epsilon= \frac{\epsilon_0}{2^{n+2\ka}C\omega_n^{1/2}}$ and $c_0 \leq \epsilon_0/(2^{n+2\ka}K(\epsilon)\omega_n^{1/2})$.  Thus, if $\omega_n c_0^{q+1} < \epsilon_0/2^{n+2\ka-1}$, then  $E(\uu_{x_0,r}, \frac{1}{2}) < \frac{\epsilon_0}{2^{n+2\ka-2}}$, which contradicts Lemma \ref{lem:nondeg-int}.
\end{proof}


\section{Homogeneous blowups and Energy decay estimates}\label{sec:blowup}
In this section we study the homogeneous rescalings and blowups. We first show that the $\ka$-homogeneous blowups exist at free boundary points.

\begin{lemma}\label{lem:blowup-exist}
Suppose $\uu$ is an almost minimizer in $B_1$ and $x_0\in \Gamma^\ka(\uu)\cap B_{1/2}$. Then for $\ka$-homogeneous rescalings $\uu_{x_0,t}(x)=\frac{\uu(x_0+tx)}{t^\ka}$, there exists $\uu_{x_0,0}\in C^1_{\loc}(\R^n)$ such that over a subsequence $t=t_j\to 0+$, $$
\uu_{x_0,t_j}\to \uu_{x_0,0}\quad\text{in }C^1_{\loc}(\R^n).
$$
Moreover, $\uu_{x_0,0}$ is a nonzero $\ka$-homogeneous global solution of $\Delta\uu_{x_0,0}=f(x_0,\uu_{x_0,0})$.
\end{lemma}

\begin{proof}
For simplicity we assume $x_0=0$ and write $\uu_t=\uu_{0,t}$ and $W(\uu,r)=W(\uu,0,0,r)$.

\medskip\noindent \emph{Step 1.} We first prove the $C^1$-convergence. For any $R>1$, Corollary~\ref{UB} ensures that there is a function $\uu_0\in C^1(B_{R/2})$ such that over a subsequence $t=t_j\to0+$, $$
\uu_{t_j}\to\uu_0\quad\text{in }C^1(B_{R/2}).
$$
By letting $R\to\infty$ and using a Cantor's diagonal argument, we obtain that for another subsequence $t=t_j\to 0+$, 
\begin{align*}
\uu_{t_j}\to\uu_0\quad\text{in }C^1_{\loc}(\R^n).
\end{align*}

\noindent \emph{Step 2.}
By the non-degeneracy in Theorem~\ref{thm:nondeg}, $\sup_{B_1}|\uu_{t_j}|\ge c_0>0$. By the $C^1$-convergence of $\uu_{t_j}$ to $\uu_0$, we have $\uu_0$ is nonzero. To show that $\uu_0$ is a global solution, for fixed $R>1$ and small $t_j$, let $\vv_{t_j}$ be the solution of $\Delta\vv_{t_j}=f(0,\vv_{t_j})$ in $B_R$ with $\vv_{t_j}=\uu_{t_j}$ on $\p B_R$. Then, by elliptic theory, \begin{align*}
    \|\vv_{t_j}\|_{C^{1,\al/2}(\overline{B_R})}\le C(n,m,R)(\|\uu_{t_j}\|_{C^{1,\al/2}(\overline{B_R})}+1)\le C.
\end{align*}
Thus, there exists a solution $\vv_0\in C^1(\overline{B_R})$ such that $$
\vv_{t_j}\to \vv_0\quad\text{in }C^1(\overline{B_R}).
$$
Moreover, we use again that $\uu_{t_j}$ is an almost minimizer with the frozen coefficients in $B_{1/2t_j}$ with a gauge function $\omega_j(\rho)\le M(t_j\rho)^\al$ to have $$
\int_{B_R}\left(|\D\uu_{t_j}|^2+2F(0,\uu_{t_j})\right)\le (1+M(t_jR)^\al)\int_{B_R}\left(|\D\vv_{t_j}|^2+2F(0,\vv_{t_j})\right).
$$
By taking $t_j\to 0$ and using the $C^1$-convergence of $\uu_{t_j}$ and $\vv_{t_j}$, we obtain 
$$
\int_{B_R}\left(|\D\uu_0|^2+2F(0,\uu_0)\right)\le \int_{B_R}\left(|\D\vv_0|^2+2F(0,\vv_0)\right).
$$
Since $\vv_{t_j}=\uu_{t_j}$ on $\p B_R$, we also have $\vv_0=\uu_0$ on $\p B_R$. This means that
$\uu_0$ is equal to the energy minimizer (or solution) $\vv_0$ in $B_R$. Since $R>1$ is arbitrary, we conclude that $\uu_0$ is a global solution.

\medskip\noindent \emph{Step 3.} Now we prove that $\uu_0$ is $\ka$-homogeneous. Fix $0<r<R<\infty$. By the Weiss-type monotonicity formula, Theorem~\ref{thm:Weiss}, we have that for small $t_j$, \begin{align}\begin{split}\label{eq:W-diff-bound}
    &W(\uu,Rt_j)-W(\uu,rt_j)\\
    &\qquad=\int_{rt_j}^{Rt_j}\frac{d}{d\rho}W(\uu,\rho)\,d\rho\\
    &\qquad\ge \int_{rt_j}^{Rt_j}\frac1{\rho^{n+2\ka}}\int_{\p B_\rho}|x\cdot\D\uu-\ka(1-b\rho^\al)\uu|^2\,dS_xd\rho\\
    &\qquad=\int_r^R\frac{t_j}{(t_j\si)^{n+2\ka}}\int_{\p B_{t_j\si}}|x\cdot\D\uu-\ka(1-b(t_j\si)^\al)\uu|^2\,dS_xd\si\\
    &\qquad=\int_r^R\frac1{t_j^{2\ka}\si^{n+2\ka}}\int_{\p B_\si}|t_jx\cdot\D\uu(t_jx)-\ka(1-b(t_j\si)^\al)\uu(t_jx)|^2\,dS_xd\si\\
    &\qquad=\int_r^R\frac1{\si^{n+2\ka}}\int_{\p B_\si}|x\cdot\D\uu_{t_j}-\ka(1-b(t_j\si)^\al)\uu_{t_j}|^2\,dS_xd\si.
\end{split}\end{align}
On the other hand, by the optimal growth estimates Theorem~\ref{thm:opt-growth}, $$
|W(\uu,r)|\le C\quad\text{for }0<r<r_0,
$$
thus $W(\uu,0+)$ is finite. Using this and taking $t_j\to 0+$ in \eqref{eq:W-diff-bound}, we get \begin{align*}
    0=W(\uu,0+)-W(\uu,0+)\ge \int_r^R\frac1{\si^{n+2\ka-1}}\int_{\p B_\si}|x\cdot\D\uu_0-\ka\uu_0|^2\,dS_xd\si.
\end{align*}
Taking $r\to0+$ and $R\to\infty$, we conclude that $x\cdot\D\uu_0-\ka\uu_0=0$ in $\R^n$, which implies that $\uu_0$ is $\ka$-homogeneous in $\R^n$.
\end{proof}

Out next objective is the polynomial decay rate of the Weiss-type energy functional $W$ at the regular free boundary points $x_0\in \mathcal{R}_\uu$, Lemma~\ref{lem:Weiss-decay}. It can be achieved with the help of the epiperimetric inequality, proved in \cite{FotShaWei21}. To describe the inequality, we let
$$
M_{x_0}(\vv):=\int_{B_1}\left(|\D\vv|^2+2F(x_0,\vv)\right)-\ka\int_{\p B_1}|\vv|^2
$$
and recall that $\Hf_{x_0}$ is a class of half-space solutions.

\begin{theorem}[Epiperimetric inequality]\label{epi}
There exist $\eta\in(0,1)$ and $\de>0$ such that if $\cc\in W^{1,2}(B_1)$ is a homogeneous function of degree $\ka$ and $\|\cc-\h\|_{W^{1,2}(B_1)}\le \de$ for some $\h\in\Hf_{x_0}$, then there exists a function $\vv\in W^{1,2}(B_1)$ such that $\vv=\cc$ on $\p B_1$ and $M_{x_0}(\vv)\le (1-\eta)M_{x_0}(\cc)+\eta M_{x_0}(\h)$.
\end{theorem}

For $x_0\in B_{1/2}$ and $0<r<1/2$, we denote the $\ka$-homogeneous replacement of $\uu$ in $B_r(x_0)$ (or equivalently, the $\ka$-homogeneous replacement of $\uu_{x_0,r}$ in $B_1$) by $$
\cc_{x_0,r}(x):=|x|^\ka\uu_{x_0,r}\left(\frac x{|x|}\right)=\left(\frac{|x|}r\right)^\ka\uu\left(x_0+\frac{r}{|x|}x\right),\quad x\in\R^n.
$$

\begin{lemma}
\label{lem:Weiss-decay}
Let $\uu$ be an almost minimizer in $B_1$ and $x_0\in \mathcal{R}_\uu\cap B_{1/2}$. Suppose that the epiperimetric inequality holds with $\eta\in(0,1)$ for each $\cc_{x_0,r},$ $0<r<r_1<1$.
Then $$
W(\uu,x_0,x_0,r)-W(\uu,x_0,x_0,0+)\le Cr^\de,\quad 0<r<r_0
$$
for some $\de=\de(n,\al,\ka,\eta)>0$.
\end{lemma}

\begin{proof}
For simplicity we may assume $x_0=0$ and write $\uu_r=\uu_{0,r}$, $\cc_r=\cc_{0,r}$ and $W(\uu,r)=W(\uu,0,0,r)$. We define \begin{align*}
    e(r)&:=W(\uu,r)-W(\uu,0+)\\
    &=\frac{e^{ar^\al}}{r^{n+2\ka-2}}\int_{B_r}\left(|\D\uu|^2+2F(0,\uu)\right)-\frac{\ka(1-br^\al)e^{ar^\al}}{r^{n+2\ka-1}}\int_{\p B_r}|\uu|^2-W(\uu,0+),
\end{align*}
and compute \begin{align*}
    \frac{d}{dr}\left(\frac{e^{ar^\al}}{r^{n+2\ka-2}}\right)=-\frac{(n+2\ka-2)(1-Mr^\al)e^{ar^\al}}{r^{n+2\ka-1}}
\end{align*}
and \begin{align*}
    \frac{d}{dr}\left(\frac{\ka(1-br^\al)e^{ar^\al}}{r^{n+2\ka-1}}\right)=\frac{-\ka(1-br^\al)e^{ar^\al}(n+2\ka-1+O(r^\al))}{r^{n+2\ka}}.
\end{align*}
Then \begin{align*}
    e'(r)&=-\frac{(n+2\ka-2)(1-Mr^\al)e^{ar^\al}}{r^{n+2\ka-1}}\int_{B_r}\left(|\D\uu|^2+2F(0,\uu)\right)\\
    &\qquad+\frac{e^{ar^\al}}{r^{n+2\ka-2}}\int_{\p B_r}\left(|\D\uu|^2+2F(0,\uu)\right)\\
    &\qquad+\frac{\ka(1-br^\al)e^{ar^\al}(n+2\ka-1+O(r^\al))}{r^{n+2\ka}}\int_{\p B_r}|\uu|^2\\
    &\qquad-\frac{\ka(1-br^\al)e^{ar^\al}}{r^{n+2\ka-1}}\int_{\p B_r}\left(2\uu\p_\nu\uu+\frac{n-1}{r}|\uu|^2\right)\\
    &\ge -\frac{n+2\ka-2}r\left(e(r)+\frac{\ka(1-br^\al)e^{ar^\al}}{r^{n+2\ka-1}}\int_{\p B_r}|\uu|^2+W(\uu,0+)\right)\\
    &\qquad+\frac{e^{ar^\al}}r\bigg[\frac1{r^{n+2\ka-3}}\int_{\p B_r}\left(|\D\uu|^2+2F(0,\uu)\right)+\frac{2\ka^2+O(r^\al)}{r^{n+2\ka-1}}\int_{\p B_r}|\uu|^2\\
    &\qquad\qquad\qquad-\frac{2\ka(1-br^\al)}{r^{n+2\ka-2}}\int_{\p B_r}\uu\p_\nu\uu\bigg]\\
    &\ge-\frac{n+2\ka-2}r\left(e(r)+W(\uu,0+)\right)\\
    &\qquad+\frac{(1-br^\al)e^{ar^\al}}r\bigg[\frac1{r^{n+2\ka-3}}\int_{\p B_r}\left(|\D\uu|^2+2F(0,\uu)\right)\\
    &\qquad\qquad\qquad\qquad+\frac{\ka(2-n)+O(r^\al)}{r^{n+2\ka-1}}\int_{\p B_r}|\uu|^2-\frac{2\ka}{r^{n+2\ka-2}}\int_{\p B_r}\uu\p_\nu\uu\bigg].
\end{align*}
To simplify the last term, we observe that $\uu_r=\cc_r$ and $\p_\nu\cc_r=\ka\cc_r$ on $\p B_1$ and that $|\D\cc_r|^2+2F(0,\cc_r)$ is homogeneous of degree $2\ka-2$, and obtain
 \begin{align*}
    &\int_{\p B_r}\left(\frac1{r^{n+2\ka-3}}\left(|\D\uu|^2+2F(0,\uu)\right)+\frac{\ka(2-n)+O(r^\al)}{r^{n+2\ka-1}}|\uu|^2-\frac{2\ka}{r^{n+2\ka-2}}\uu\p_\nu\uu\right)\\
    &=\int_{\p B_1}\left(|\D\uu_r|^2+2F(0,\uu_r)+(\ka(2-n)+O(r^\al))|\uu_r|^2-2\ka\uu_r\p_\nu\uu_r\right)\\
    &=\int_{\p B_1}\left(|\p_\nu\uu_r-\ka\uu_r|^2+|\p_\theta\uu_r|^2+2F(0,\uu_r)-\left(\ka(n+\ka-2)+O(r^\al)\right)|\uu_r|^2\right)\\
    &\ge \int_{\p B_1}\left(|\p_\theta\cc_r|^2+2F(0,\cc_r)-\left(\ka(n+\ka-2)+O(r^\al)\right)|\cc_r|^2\right)\\
    &=\int_{\p B_1}\left(|\D\cc_r|^2+2F(0,\cc_r)-\left(\ka(n+2\ka-2)+O(r^\al)\right)|\cc_r|^2\right)\\
    &=(n+2\ka-2)\left[\int_{B_1}\left(|\D\cc_r|^2+2F(0,\cc_r)\right)-(\ka+O(r^\al))\int_{\p B_1}|\cc_r|^2\right]\\
    &=(n+2\ka-2)M_{0}(\cc_r)+O(r^\al)\int_{\p B_1}|\uu_r|^2.
\end{align*}
Thus \begin{align}\label{eq:diff-ineq}\begin{split}
    e'(r)&\ge-\frac{n+2\ka-2}r\left(e(r)+W(\uu,0+)\right)+\frac{(1-br^\al)e^{ar^\al}(n+2\ka-2)}rM_{0}(\cc_r)\\
    &\qquad+\frac{O(r^\al)}r\int_{\p B_1}|\uu_r|^2.
\end{split}\end{align}
We want to estimate $M_0(\cc_r)$. From the assumption that the epiperimetric inequality holds for $\cc_r$, we have $M_{0}(\vv^r)\le(1-\eta)M_{0}(\cc_r)+\eta M_{0}(\h^r)$ for some $\h^r\in\Hf_0$ and $\vv^r\in W^{1,2}(B_1)$ with $\vv^r=\cc_r$ on $\p B_1$. In addition, $0\in\mathcal{R}_\uu$ ensures that there exists a sequence $t_j\to 0+$ such that $\uu_{t_j}\to\h_0$ in $C^1_{\loc}(\R^n)$ for some $\h_0\in \Hf_0$. Then
\begin{align*}
    W(\uu,0+)&=\lim_{j\to\infty}W(\uu,t_j)\\
    &=\lim_{j\to\infty}e^{at_j^\al}\left[\int_{B_1}\left(|\D\uu_{t_j}|^2+2F(0,\uu_{t_j})\right)-\ka(1-bt_j^\al)\int_{\p B_1}|\uu_{t_j}|^2\right]\\
    &=M_{0}(\h_{0})=M_{0}(\h^r).
\end{align*}
Here, the last equality holds since both $\h_j$ and $\h^r$ are elements in $\Hf_0$. By the epiperimetric inequality and the almost-minimality of $\uu_r$,
\begin{align*}
    (1-\eta)M_{0}(\cc_r)+\eta W(\uu,0+)
    &\ge M_{0}(\vv^r)=\int_{B_1}\left(|\D\vv^r|^2+2F(0,\vv^r)\right)-\ka\int_{\p B_1}|\vv^r|^2\\
    &\ge \frac1{1+Mr^\al}\int_{B_1}\left(|\D\uu_r|^2+2F(0,\uu_r)\right)-\ka\int_{\p B_1}|\uu_r|^2\\
    &=\frac{e^{-ar^\al}}{1+Mr^\al}W(\uu,r)+O(r^\al)\int_{\p B_1}|\uu_r|^2.
    \end{align*}
We rewrite it as \begin{align*}
    M_{0}(\cc_r)\ge\frac{\frac{e^{-ar^\al}}{1+Mr^\al}W(\uu,r)-\eta W(\uu,0+)}{1-\eta}+O(r^\al)\int_{\p B_1}|\uu_r|^2
\end{align*}
and, combining this with \eqref{eq:diff-ineq}, obtain 
\begin{align*}
    e'(r)&\ge-\frac{n+2\ka-2}r\left(e(r)+W(\uu,0+)\right)\\
    &\qquad+\frac{(1-br^\al)e^{ar^\al}(n+2\ka-2)}r\left(\frac{\frac{e^{-ar^\al}}{1+Mr^\al}W(\uu,r)-\eta W(\uu,0+)}{1-\eta}\right)\\
    &\qquad+\frac{O(r^\al)}r\int_{\p B_1}|\uu_r|^2.
  \end{align*}
Note that from Theorem~\ref{thm:opt-growth}, there is a constant $C>0$ such that $$
W(\uu,0+)\le W(\uu,r)\le C
$$
and $$
\int_{\p B_1}|\uu_r|^2=\frac1{r^{n+2\ka-1}}\int_{\p B_r}|\uu|^2\le C
$$
for small $r>0$. Then \begin{align*}
    e'(r)&\ge -\frac{n+2\ka-2}r\left(e(r)+W(\uu,0+)\right)+\frac{n+2\ka-2}r\left(\frac{W(\uu,r)-\eta W(\uu,0+)}{1-\eta}\right)\\
    &\qquad+\frac{O(r^\al)}r\left(W(\uu,r)+W(\uu,0+)+\int_{\p B_1}|\uu_r|^2\right)\\
    &\ge -\frac{n+2\ka-2}re(r)+\frac{n+2\ka-2}r\left(\frac{W(\uu,r)-W(\uu,0+)}{1-\eta}\right)-C_1r^{\al-1}\\
    &=\left(\frac{(n+2\ka-2)\eta}{1-\eta}\right)\frac{e(r)}r-C_1r^{\al-1}.
\end{align*}
Now, take $\de=\de(n,\al,\ka,\eta)$ such that $0<\de<\min\left\{\frac{(n+2\ka-2)\eta}{1-\eta},\al \right\}$. Using the differential inequality above for $e(r)$ and that $e(r)=W(\uu,r)-W(\uu,0+)\ge 0$, we have for $0<r<r_0$ \begin{align*}
    \frac{d}{dr}\left[e(r)r^{-\de}+\frac{C_1}{\al-\de}r^{\al-\de}\right]&=r^{-\de}\left[e'(r)-\frac\de re(r)\right]+C_1r^{\al-\de-1}\\
    &\ge r^{-\de}\left[\left(\frac{(n+2\ka-2)\eta}{1-\eta}-\de\right)\frac{e(r)}r-C_1r^{\al-1}\right]+C_1r^{\al-\de-1}\\
    &\ge 0.
\end{align*}
Thus $$
e(r)r^{-\de}\le e(r_0)r_0^{-\de}+\frac{C_0}{\al-\de}r_0^{\al-\de},
$$
and hence we conclude that \begin{equation*}
W(\uu,r)-W(\uu,0+)=e(r)\le Cr^\de.
\qedhere\end{equation*}
\end{proof}

Now, we consider an auxiliary
function
$$
\phi(r):=e^{-(\ka b/\alpha)r^\alpha}r^{\ka},\quad
r>0,
$$
which is a solution of the differential equation
$$
\phi'(r)=\ka\,\phi(r)\frac{1-b r^\alpha}r,\quad r>0.
$$
For $x_0\in B_{1/2}$, we define the \emph{$\ka$-almost homogeneous rescalings} by
$$
\uu_{x_0, r}^{\phi}(x):=\frac{\uu(rx+x_0)}{\phi(r)},\quad x\in B_{1/(2r)}.
$$

\begin{lemma}
[Rotation estimate]\label{lem:rot-est}
Under the same assumption as in Lemma~\ref{lem:Weiss-decay},
$$
\int_{\p B_1}|\uu_{x_0,t}^\phi-\uu_{x_0,s}^\phi|\le Ct^{\de/2},\quad s<t<t_0.
$$
\end{lemma}

\begin{proof}
Without loss of generality, assume $x_0=0$. For $\uu_r^\phi=\uu_{0,r}^\phi$ and $W(\uu,r)=W(\uu,0,0,r)$,
  \begin{align*}
    \frac{d}{dr}\uu_{r}^{\phi}(x)
    &=\frac{\nabla \uu(rx)\cdot x}{\phi(r)}-\frac{\uu(rx)[\phi'(r)/\phi(r)]}{\phi(r)}\\
    &=\frac{1}{\phi(r)}\left(\nabla \uu(rx)\cdot
      x-\frac{\ka(1-br^\alpha)}{r}\uu(rx)\right).
  \end{align*}
By Theorem~\ref{thm:Weiss}, we have for $0<r<t_0$ 
\begin{align*}
  &\left(\int_{\partial
    B_1}\left[\frac{d}{dr}\uu_r^\phi(\xi)\right]^2dS_\xi\right)^{1/2}\\
  &=
    \frac{1}{\phi(r)}\left(\int_{\partial B_1}\left|\nabla \uu(r\xi)\cdot \xi
    -\frac{\ka(1-br^\alpha)}{r}\uu(r \xi)\right|^2  dS_\xi\right)^{1/2}\\
  &=\frac{1}{\phi(r)}\left(\frac{1}{r^{n-1}}\int_{\partial
    B_r}\left|\D \uu(x)\cdot \nu-\frac{\ka(1-br^\alpha)}{r}
    \uu(x)\right|^2dS_x\right)^{1/2}\\
  &\leq\frac{1}{\phi(r)}\left(\frac{1}{r^{n-1}}\frac{r^{n+2\ka-2}}{e^{a
    r^\alpha}}\frac{d}{dr}W(\uu,r)\right)^{1/2}=\frac{e^{c
    r^\alpha}}{r^{1/2}}\left(\frac{d}{dr}W(\uu,r)\right)^{1/2}
    ,\quad c=\frac{\ka b}{\alpha}-\frac{a}{2}.
\end{align*}
Using this and Lemma~\ref{lem:Weiss-decay}, we can compute \begin{align*}
    \int_{\p B_1}|\uu_{t}^\phi-\uu_{s}^\phi|&\le \int_{\p B_1}\int_s^t\left|\frac{d}{dr}\uu_{r}^\phi\right|\,dr=\int_s^t\int_{\p B_1}\left|\frac{d}{dr}\uu_{r}^\phi\right|\,dr\\
    &\le C_n\int_s^t\left(\int_{\p B_1}\left|\frac{d}{dr}\uu_{r}^\phi\right|^2\right)^{1/2}\,dr\\
    &\le C_n\left(\int_s^t\frac1r\,dr\right)^{1/2}\left(\int_s^tr\int_{\p B_1}\left|\frac{d}{dr}\uu_{r}^\phi\right|^2\,dr\right)^{1/2}\\
    &\le C_ne^{ct^\al}\left(\log\frac ts\right)^{1/2}\left(\int_s^t\frac{d}{dr}W(\uu,r)\,dr\right)^{1/2}\\
    &\le C\left(\log\frac ts\right)^{1/2}(W(\uu,t)-W(\uu,s))^{1/2}\\
    &\le C\left(\log\frac ts\right)^{1/2}t^{\de/2},\quad 0<t<t_0.
\end{align*}
Now, by a standard dyadic argument, we conclude that \begin{equation*}
\int_{\p B_1}|\uu_{t}^\phi-\uu_{s}^\phi|\le Ct^{\de/2}.
\qedhere\end{equation*}
\end{proof}

The following are generalization of Lemma~4.4 and Proposition~4.6 in \cite{FotShaWei21} on $\ka$-homogeneous solutions from the case $\la_{\pm}=1$ to the constant coefficients $\la_\pm(x_0)$, which can be proved in similar fashion.

\begin{lemma}\label{lem:H-isolated}
For every $x_0\in B_{1/2}$, $\Hf_{x_0}$ is isolated (in the topology of $W^{1,2}(B_1)$) within the class of $\ka$-homogeneous solutions $\vv$ for $\Delta\vv=f(x_0,\vv)$.
\end{lemma}

\begin{proposition}\label{prop:classi-hom-sol}
For $x_0\in B_{1/2}$, let $\vv\not\equiv0$ be a $\ka$-homogeneous solution of $\Delta\vv=f(x_0,\vv)$ satisfying $\{|\vv|=0\}^{\mathrm{0}}\neq\emptyset$. Then $M_{x_0}(\vv)\ge \mathcal{B}_{x_0}$ and equality implies $\vv\in\Hf_{x_0}$; here $\mathcal{B}_{x_0}=M_{x_0}(\h)$ for every $\h\in\Hf_{x_0}$.
\end{proposition}

The proof of the next proposition can be obtained as in Proposition~4.5 in \cite{AleFotShaWei22}, by using Lemma~\ref{lem:H-isolated} and a continuity argument.

\begin{proposition}
\label{prop:all-blowup-R}
If $x_0\in\mathcal{R}_\uu$, then all blowup limits of $\uu$ at $x_0$ belong to $\Hf_{x_0}$.
\end{proposition}


\section{Regularity of the regular set}\label{sec:reg-set}

In this last section we establish one of the main result in this paper, the $C^{1,\gamma}$-regularity of the regular set $\mathcal{R}_\uu$.

We begin by showing that $\mathcal{R}_\uu$ is an open set in $\Gamma(\uu)$.

\begin{lemma}\label{lem:rel-open}
$\mathcal{R}_\uu$ is open relative to $\Gamma(\uu)$.
\end{lemma}

\begin{proof}
\emph{Step 1.} For points $y\in B_{1/2}$, we let $\mathbb A_{y}$ be a set of $\ka$-homogeneous solutions $\vv$ of $\Delta\vv=f(y,\vv)$ satisfying $\vv\not\in\Hf_{y}$, and define $$
\rho_{y}:=\inf\{\|\h-\vv\|_{C^{1}(\overline{B_1})}\,:\,\h\in\Hf_{y},\vv\in \mathbb A_{y}\}.
$$
From $\|\h-\vv\|_{C^1(\overline{B_1})}\ge c(n)\|\h-\vv\|_{W^{1,2}(B_1)}$, the isolation property of $\Hf_{x_0}$ in Lemma~\ref{lem:H-isolated} also holds in $C^1(\overline{B_1})$-norm, thus $\rho_y>0$ for every $y\in B_{1/2}$.

We claim that there is a universal constant $c_1>0$ such that $c_1<\rho_y<1/c_1$ for all $y\in B_{1/2}$. Indeed, the second inequality $\rho_y<1/c_1$ is obvious. For the first one, we assume to the contrary that $\rho_{y_i}\to0$ for a sequence $y_i\in B_{1/2}$. This gives sequences $\h_{y_i}\in\Hf_{y_i}$ and $\vv_{y_i}\in\mathbb A_{y_i}$ such that $\dist(\h_{y_i},\vv_{y_i})\to0$, where the distance is measured in $C^1(\overline{B_1})$-norm. Over a subsequence, we have $y_i\to y_0$ and, using that $\h_{y_i}$ and $\vv_{y_i}$ are uniformly bounded, $\h_{y_i}\to\h_{y_0}$ and $\vv_{y_i}\to\vv_{y_0}$ for some $\h_{y_0}\in\Hf_{y_0}$ and $\vv_{y_0}\in\mathbb A_{y_0}$. It follows that $\dist(\h_{y_0},\vv_{y_0})=\lim_{i\to\infty}\dist(\h_{y_i},\vv_{y_i})=0$, meaning that $\h_{y_0}=\vv_{y_0}$, a contradiction.

\medskip\noindent\emph{Step 2.} 
Towards a contradiction we assume that the statement of Lemma~\ref{lem:rel-open} is not true. Then we can find a point $x_0\in\mathcal{R}_\uu$ and a sequence of points $x_i\in\Gamma(\uu)\setminus\mathcal R_\uu$ such that $x_i\to x_0$.

For a small constant $\e_1=\e_1(n,q,M,c_1)>0$ to be specified later, we claim that there is a sequence $r_i\to 0$ and a subsequence of $x_i$, still denoted by $x_i$, such that \begin{align}
    \label{eq:dist-H}
    \dist(\uu_{x_i,r_i},\Hf_{x_0})=\e_1,
\end{align}
where $\uu_{x_i,r_i}(x)=\frac{\uu(x_i+r_ix)}{r_i^\ka}$ are the $\ka$-homogeneous rescalings. 

Indeed, since $x_0\in\mathcal R_\uu$, we can find a sequence $t_j\to0$ such that \begin{align}\label{eq:u-H-dist-1}
\dist(\uu_{x_0,t_j},\Hf_{x_0})<\e_1/2.
\end{align}
For each fixed $t_j$, we take a point from the above sequence $\{x_i\}_{i=1}^\infty$ (and redefine the point as $x_j$ for convenience) close to $x_0$ such that
\begin{align*}
    \dist(\uu_{x_j,t_j},\Hf_{x_0})\le \dist(\uu_{x_j,t_j},\uu_{x_0,t_j})+\dist(\uu_{x_0,t_j},\Hf_{x_0})<\e_1.
\end{align*}
On the other hand, using $x_j\not\in\mathcal{R}_\uu$ and Proposition~\ref{prop:all-blowup-R}, we see that there is $\tau_j<t_j$ such that $\dist(\uu_{x_j,\tau_j},\Hf_{x_j})>\rho_{x_j}/2$. Using the result in \emph{Step 1}, we have $\dist(\uu_{x_j,\tau_j},\Hf_{x_j})>c_1/2>2\e_1$ for small $\e_1$. We next want to show that for large $j$ \begin{align}
    \label{eq:u-H-dist-2}
    \dist(\uu_{x_j,\tau_j},\Hf_{x_0})\ge (3/2)\e_1.
\end{align}
For this aim, we let $\h_{x_0}\in \Hf_{x_0}$ be given. Then $\h_{x_j}:=\frac{\be_{x_j}}{\be_{x_0}}\h_{x_0}\in\Hf_{x_j}$ and $\frac{\be_{x_j}}{\be_{x_0}}=\left(\frac{\la_+(x_j)}{\la_+(x_0)}\right)^{\ka/2}\to1$ as $j\to\infty$. Thus $\dist(\h_{x_j},\h_{x_0})<\e_1/2$ for large $j$, and hence \begin{align*}
    \dist(\uu_{x_j,\tau_j},\h_{x_0})&\ge\dist(\uu_{x_j,\tau_j},\h_{x_j})-\dist(\h_{x_j},\h_{x_0})\\
    &>2\e_1-\e_1/2=(3/2)\e_1.
\end{align*}
This implies \eqref{eq:u-H-dist-2}.

Now, \eqref{eq:u-H-dist-1} and \eqref{eq:u-H-dist-2} ensure the existence of $r_j\in(\tau_j,t_j)$ such that $$
\dist(\uu_{x_j,r_j},\Hf_{x_0})=\e_1.
$$

\medskip\noindent\emph{Step 3.} \eqref{eq:dist-H} implies that $\{\uu_{x_i,r_i}\}$ is uniformly bounded in $C^1$-norm (so in $C^{1,\al/2}$-norm as well), thus we can follow the argument in \emph{Step 1-2} in the proof of Lemma~\ref{lem:blowup-exist} with $\uu_{x_i,r_i}$ in the place of $\uu_t$ to have that over a subsequence $$
\uu_{x_i,r_i}\to\uu^*\quad\text{in }C^1_{\loc}(\R^n)
$$
for some nonzero global solution $\uu^*$ of $\Delta\uu^*=f(x_0,\uu^*)$.

From \eqref{eq:dist-H} again, we may assume that $\|\uu_{x_i,r_i}-\h\|_{C^1(\overline{B_1})}\le 2\e_1$ for $\h(x)=\be_{x_0}(x^1_+)^\ka$. Then $|\uu_{x_i,r_i}|+|\D\uu_{x_i,r_i}|\le 2\e_1$ in $B_1\cap\{x_1\le 0\}$. By the nondegeneracy property, Lemma~\ref{lem:nondeg-int}, we know that $\int_{B_t(z)}\left(|\D\uu_{x_i,r_i}|^2+|\uu_{x_i,r_i}|^{q+1}\right)\ge \e_0t^{n+2\ka-2}$ for any $z\in\overline{\{|\uu_{x_i,r_i}|>0\}}$ and $B_t(z)\Subset B_1$. We can deduce from the preceding two inequalities that if $\e_1$ is small, then $\uu_{x_i,r_i}\equiv0$ in $B_1\cap\{x_1\le-1/4\}$. Therefore, the coincidence set $\{|\uu^*|=0\}$ has a nonempty interior, and there exist an open ball $D\Subset\{|\uu^*|=0\}$ and a point $z_0\in\p D\cap\p\{|\uu^*|>0\}$.

\medskip\noindent\emph{Step 4.} We claim that $\|\uu^*\|_{L^\infty(B_r(z_0))}=O(r^\ka)$.

In fact, we can proceed as in the proof of the sufficiency part of Theorem~1.2 in \cite{FotShaWei21}. In the theorem, they assume $\|\uu\|_{L^\infty(B_r)}=o(r^{\lfloor\ka\rfloor})$ and prove $\|\uu\|_{L^\infty(B_r)}=O(r^\ka)$. We want to show that the condition $\|\uu\|_{L^\infty(B_r)}=o(r^{\lfloor\ka\rfloor})$ can be replaced by $0\in\p D\cap\p\{|\uu|>0\}$, where $D$ is an open ball contained in $\{|\uu|=0\}$ (then it can be applied to $\uu^*$ in our case, and $\|\uu^*\|_{L^\infty(B_r(z_0))}=O(r^\ka)$ follows). Indeed, the growth condition on $\uu$ in the theorem is used only to prove the following: if $\tilde\uu_j(x)=P_j(x)+\Gamma_j(x)$ in $B_r$, where $\tilde\uu_j(x)=\frac{\uu(rx)}{jr_j^\ka}$ is a rescaling of $\uu$ at $0$, $P_j$ is a harmonic polynomial of degree $l\le \lfloor\ka\rfloor$ and $|\Gamma_j(x)|\le C|x|^{l+\e}$, $0<\e<1$, then $P_j\equiv0$. To see that $0\in\p D\cap\p\{|\uu|>0\}$ also gives the same result, we observe that it implies that $\tilde\uu_j=0$ in an open subset $A_r$ of the ball $B_r$. This is possible only when $P_j=\Gamma_j=0$ in $A_r$. Thus, by the unique continuation $P_j\equiv0$.

\medskip\noindent\emph{Step 5.} Consider the standard Weiss energy $$
W^0(\uu^*,z,y,s):=\frac1{s^{n+2\ka-2}}\int_{B_s(z)}\left(|\D\uu^*|^2+2F(y,\uu^*)\right)-\frac\ka{s^{n+2\ka-1}}\int_{\p B_s(z)}|\uu^*|^2.
$$
By the result in \emph{Step 4}, there exists a $\ka$-homogeneous blowup $\uu^{**}$ of $\uu^*$ at $z_0$ (i.e., $\uu^{**}(x)=\lim_{r\to0}\frac{\uu^*(rx+z_0)}{r^\ka}$ over a subsequence). From $z_0\in\p D$, $\uu^{**}$ should have a nonempty interior of the zero-set $\{|\uu^{**}|=0\}$ near the origin $0$, thus by Proposition~\ref{prop:classi-hom-sol}
$$
W^0(\uu^*,z_0,x_0,0+)=M_{x_0}(\uu^{**})\ge\mathcal{B}_{x_0}.
$$
Moreover, we observe that for any fixed $r\le 1$ \begin{align*}
    W^0(\uu^*,z_0,x_0,0+)&\le W^0(\uu^*,z_0,x_0,r)\le W(\uu^*,z_0,x_0,r)=\lim_{i\to\infty}W(\uu,x_i+r_iz_0,x_0,rr_i)\\
    &\le \lim_{i\to\infty}W(\uu,x_i+r_iz_0,x_0,s)=W(\uu,x_0,x_0,s)
\end{align*}
for any $s>0$. Taking $s\searrow0$ and using $W(\uu,x_0,x_0,0+)=\mathcal{B}_{x_0}$, we obtain that $W^0(\uu^*,z_0,x_0,r)=\mathcal{B}_{x_0}$ for $0<r<1$. Thus $\uu^*$ is a $\ka$-homogeneous function with respect to $z_0$.

Now, we can apply Proposition~\ref{prop:classi-hom-sol} to obtain that $\uu^*$ is a half-space solution with respect to $z_0$, i.e., $\uu^*(\cdot-z_0)\in\Hf_{x_0}$. Since $\uu_{x_i,r_i}$ satisfies $|\uu_{x_i,r_i}(0)|=0$ and the nondegeneracy $\int_{B_t}\left(|\D\uu_{x_i,r_i}|^2+|\uu_{x_i,r_i}|^{q+1}\right)\ge\e_0 t^{n+2\ka-2}$, $\uu^*$ also satisfies the similar equations, and thus $z_0=0$. This implies $\uu^*\in\Hf_{x_0}$, which contradicts \eqref{eq:dist-H}.
\end{proof}

\begin{lemma}\label{lem:c-H-unif}
Let $C_h$ be a compact subset of $\mathcal{R}_\uu$. For any $\e>0$, there is $r_0>0$ such that if $x_0\in C_h$ and $0<r<r_0$, then the $\ka$-homogeneous replacement $\cc_{x_0,r}$ of $\uu$ satisfies \begin{align}\label{eq:c-H-unif}
    \dist(\cc_{x_0,r},\Hf_{x_0})<\e,
\end{align}
where the distance is measured in the $C^1(\overline{B_1})$-norm.
\end{lemma}

\begin{proof}
We claim that for any $\e>0$, there is $r_0>0$ such that $$
\dist(\uu_{x_0,r},\Hf_{x_0})<\e\quad\text{for any }x_0\in C_h\text{ and }0<r<r_0,
$$
which readily gives \eqref{eq:c-H-unif} (see the proof of Lemma~10 in \cite{DeSJeoSha21}).

Towards a contradiction we assume that there exist a constant $\e_0>0$ and sequences $x_j\in C_h$ (converging to $x_0\in C_h$) and $r_j\to0$ such that \begin{align*}
    \dist(\uu_{x_j,r_j},\Hf_{x_j})\ge\e_0.
\end{align*}
By a continuity argument, for each $\theta\in(0,1)$ we can find a sequence $t_j<r_j$ such that $$
\dist(\uu_{x_j,t_j},\Hf_{x_j})=\theta\e_0.
$$
By following the argument in \emph{Step 1-2} in the proof of Lemma~\ref{lem:blowup-exist} with $\uu_{x_j,t_j}$ in the place of $\uu_t$, we can show that up to a subsequence $$
\uu_{x_j,t_j}\to \uu_{x_0}\quad\text{in }C^1_{\loc}(\R^n;\R^m)
$$
for some nonzero global solution $\uu_{x_0}\in C^1_{\loc}(\R^n;\R^m)$ of
$\Delta\uu_{x_0}=f(x_0,\uu_{x_0})$. We remark that the blowup $\uu_{x_0}$ depends on the sequence $\{t_j\}$, thus on the choice of $0<\theta<1$.

From $\dist(\uu_{x_j,t_j},\Hf_{x_j})=\theta\e_0$, we can take $\h_{x_j}\in\Hf_{x_j}$ such that $\dist(\uu_{x_j,t_j},\h_{x_j})\le 2\theta\e_0.$ 
For each $j$ we define $\h_{x_0}^j:=\frac{\be_{x_0}}{\be_{x_j}}\h_{x_j}\in\Hf_{x_0}$. Since $\frac{\be_{x_0}}{\be_{x_j}}=\left(\frac{\la_+(x_0)}{\la_+(x_j)}\right)^{\ka/2}\to1$, $$
\dist(\h_{x_j},\h_{x_0}^j)\le o(|x_j-x_0|).
$$
Thus,\begin{align*}
    \dist(\uu_{x_0},\h_{x_0}^j)&\le \dist(\uu_{x_0},\uu_{x_j,t_j})+\dist(\uu_{x_j,t_j},\h_{x_j})+\dist(\h_{x_j},\h_{x_0}^j)\\
    &\le 2\theta\e_0+o(|x_j-x_0|),
\end{align*}
and hence \begin{align*}
    \dist(\uu_{x_0},\Hf_{x_0})\le\limsup_{j\to\infty}\dist(\uu_{x_0},\h_{x_0}^j)\le 2\theta\e_0.
\end{align*}
On the other hand, for each $h^{x_0}\in\Hf_{x_0}$ we let $\h^{x_j}:=\frac{\be_{x_j}}{\be_{x_0}}\h^{x_0}\in\Hf_{x_j}$ so that $\dist(\h^{x_j},\h^{x_0})\le o(|x_j-x_0|)$. Using $\dist(\uu_{x_j,t_j},\Hf_{x_j})=\theta\e_0$ again, we obtain \begin{align*}
    \dist(\uu_{x_0},\h^{x^0})&\ge\dist(\uu_{x_j,t_j},\h^{x_j})-\dist(\h^{x_j},\h^{x_0})-\dist(\uu_{x_j,t_j},\uu_{x_0})\\
    &\ge\theta\e_0-o(|x_j-x_0|),
\end{align*}
and conclude $$
\dist(\uu_{x_0},\Hf_{x_0})\ge\theta\e_0.
$$
Therefore, $$
\theta\e_0\le\dist(\uu_{x_0},\Hf_{x_0})\le2\theta\e_0.
$$
For $\theta>0$ small enough, this inequality contradicts the isolation property of $\Hf_{x_0}$ in Lemma~\ref{lem:H-isolated}, provided $\uu_{x_0}$ is homogeneous of degree $\ka$. 

To prove the homogeneity, we fix $0<r<R<\infty$ and follow the argument in \emph{Step 3} in Lemma~\ref{lem:blowup-exist} to obtain \begin{align}
    \label{eq:weiss-diff}
    \begin{split}
    &W(\uu,x_j,x_j,Rt)-W(\uu,x_j,x_j,rt)\\
    &\qquad\ge\int_r^R\frac1{\si^{n+2\ka}}\int_{\p B_\si}\left|x\cdot\D\uu_{x_j,t}-\ka\left(1-b(t\si)^\al\right)\uu_{x_j,t}\right|^2\,dS_xd\si
\end{split}\end{align}
for small $t$. Recall the standard Weiss energies $$
W^0(\uu,x_j,x_j,t)=\frac1{t^{n+2\ka-2}}\int_{B_t(x_j)}\left(|\D\uu|^2+2F(x_j,\uu)\right)-\frac\ka{t^{n+2\ka-1}}\int_{\p B_t(x_j)}|\uu|^2.
$$
We have $$
M_{x_j}(\uu_{x_j,t})=W^0(\uu,x_j,x_j,t)=W(\uu,x_j,x_j,t)+O(t^\al),
$$
where the second equality holds by Theorem~\ref{thm:opt-growth}. This, together with $x_j\in\cR_\uu$ and the monotonicity of $W(\uu,x_j,x_j,\cdot)$, gives that $$
W(\uu,x_j,x_j,t)\searrow\cB_{x_j}\quad\text{as }t\searrow0.
$$
Applying Dini's theorem gives that for any $\e>0$ there exists $t_0=t_0(\e)>0$ such that $$
\cB_{x_j}\le W(\uu,x_j,x_j,t)\le\cB_{x_j}+\e\quad\text{for any }t<t_0\text{ and }x_j\in C_h.
$$
Then, for large $j$, we have $\cB_{x_j}\le W(\uu,x_j,x_j,Rt_j)\le\cB_{x_j}+\e$, and thus $$
\cB_{x_0}\le\liminf_{j\to\infty}W(\uu,x_j,x_j,Rt_j)\le\limsup_{j\to\infty}W(\uu,x_j,x_j,Rt_j)\le\cB_{x_0}+\e.
$$
Taking $\e\searrow0$, we obtain that $$
\lim_{j\to\infty}W(\uu,x_j,x_j,Rt_j)=\cB_{x_0}.
$$
Similarly, we have $\lim_{j\to\infty}W(\uu,x_j,x_j,rt_j)=\cB_{x_0}$. They enable us to take $j\to\infty$ in \eqref{eq:weiss-diff} to get $$
0=\int_r^R\frac1{\si^{n+2\ka}}\int_{\p B_\si}|x\cdot\D\uu_{x_0}-\ka\uu_{x_0}|^2\,dS_xd\si.
$$
Since $0<r<R<\infty$ are arbitrary, we conclude that $x\cdot\D\uu_{x_0}-\ka\uu_{x_0}=0$, or $\uu_{x_0}$ is $\ka$-homogeneous in $\R^n$.
\end{proof}

\begin{lemma}\label{lem:rescaling-blowup-est}
Let $\uu$ be an almost minimizer in $B_1$, $C_h$ a compact subset of $\mathcal{R}_\uu$, and $\de$ as in Lemma~\ref{lem:Weiss-decay}. Then for every $x_0\in C_h$ there is a unique blowup $\uu_{x_0,0}\in\Hf_{x_0}$. Moreover, there exist $t_0>0$ and $C>0$ such that $$
\int_{\p B_1}|\uu^\phi_{x_0,t}-\uu_{x_0,0}|\le Ct^{\de/2}
$$
for all $0<t<t_0$ and $x_0\in C_h$.
\end{lemma}

\begin{proof}
By Lemma~\ref{lem:rot-est} and Lemma~\ref{lem:c-H-unif}, $$
\int_{\p B_1}|\uu_{x_0,t}^\phi-\uu_{x_0,s}^\phi|\le Ct^{\de/2},\quad s<t<t_0.
$$
By the definition of $\mathcal{R}_\uu$, for a subsequence of $t_j\to0+$ we have $\uu_{x_0,t_j}\to\uu_{x_0,0}\in\Hf_{x_0}$. From $\lim_{t_j\to0+}\frac{\phi(t)}{t^\ka}=1$, we also have $\uu_{x_0,t_j}^\phi\to\uu_{x_0,0}$. Taking $s=t_j$ in the above inequality and passing to the limit, we get $$
\int_{\p B_1}|\uu_{x_0,t}^\phi-\uu_{x_0,0}|\le Ct^{\de/2},\quad t<t_0.
$$
To prove the uniqueness, we let $\tilde{\uu}_{x_0,0}$ be another blowup. Then $$
\int_{\p B_1}|\tilde{\uu}_{x_0,0}-\uu_{x_0,0}|=0.
$$
We proved in Lemma~\ref{lem:blowup-exist} that every blowup is $\ka$-homogeneous in $\R^n$, thus $\tilde{\uu}_{x_0,0}=\uu_{x_0,0}$ in $\R^n$.
\end{proof}

As a consequence of the previous results, we prove the regularity of $\mathcal R_\uu$.

\begin{proof}[Proof of Theorem~\ref{thm:reg-set}]
The proof of the theorem is similar to those of Theorem~3 in \cite{DeSJeoSha21} and Theorem~1.4 in \cite{FotShaWei21}.

\medskip\noindent\emph{Step 1.} Let $x_0\in\mathcal{R}_\uu$. By Lemma~\ref{lem:rel-open} and Lemma~\ref{lem:rescaling-blowup-est}, there exists $\rho_0>0$ such that $B_{2\rho_0}(x_0)\subset B_1$, $B_{2\rho_0}(x_0)\cap\Gamma^\ka(\uu)=B_{2\rho_0}(x_0)\cap\mathcal{R}_\uu$ and $$
\int_{\p B_1}|\uu_{x_1,r}^\phi-\be_{x_1}\max(x\cdot\nu(x_1),0)^\ka\mathbf{e}(x_1)|\le Cr^{\de/2},
$$
for any $x_1\in \Gamma^\ka(\uu)\cap \overline{B_{\rho_0}(x_0)}$ and for any $0<r<\rho_0$. We then claim that $x_1\longmapsto\nu(x_1)$ and $x_1\longmapsto\mathbf{e}(x_1)$ are H\"older continuous of order $\gamma$ on $\Gamma^\ka(\uu)\cap\overline{B_{\rho_1}(x_0)}$ for some $\gamma=\gamma(n,\al,q,\eta)>0$ and $\rho_1\in(0,\rho_0)$. Indeed, we observe that for $x_1$ and $x_2$ near $x_0$ and for small $r>0$,
\begin{align*}
    &\int_{\p B_1}|\be_{x_1}\max(x\cdot\nu(x_1),0)^\ka\mathbf{e}(x_1)-\be_{x_2}\max(x\cdot\nu(x_2),0)^\ka\mathbf{e}(x_2)|\,dS_x\\
    &\qquad\le 2Cr^{\de/2}+\int_{\p B_1}|\uu_{x_1,r}^\phi-\uu_{x_2,r}^\phi|\\
    &\qquad\le 2Cr^{\de/2}+\frac1{\phi(r)}\int_{\p B_1}\int_0^1|\D\uu(rx+(1-t)x_1+tx_2)||x_1-x_2|\,dt\,dS_x\\
    &\qquad\le 2Cr^{\de/2}+C\frac{|x_1-x_2|}{r^\ka}.
\end{align*}
Moreover, \begin{align*}
    &\int_{\p B_1}|\be_{x_1}\max(x\cdot\nu(x_2),0)^\ka\mathbf{e}(x_2)-\be_{x_2}\max(x\cdot\nu(x_2),0)^\ka\mathbf{e}(x_2)|\,dS_x\\
    &\qquad\le |\be_{x_1}-\be_{x_2}|\int_{\p B_1}\max(x\cdot\nu(x_2),0)^\ka|\mathbf{e}(x_2)|\\
    &\qquad\le C|\la_+(x_1)^{\ka/2}-\la_+(x_2)^{\ka/2}|\le C|x_1-x_2|^\al.
\end{align*}
The above two estimates give \begin{align*}\begin{multlined}
    \be_{x_1}\int_{\p B_1}|\max(x\cdot\nu(x_1),0)^\ka\mathbf{e}(x_1)-\max(x\cdot\nu(x_2),0)^\ka\mathbf{e}(x_2)|\,dS_x\\
    \le C\left(r^{\de/2}+\frac{|x_1-x_2|}{r^\ka}+|x_1-x_2|^\al\right).
\end{multlined}\end{align*}
Taking $r=|x_1-x_2|^{\frac2{\de+2\ka}}$, we get $$
\int_{\p B_1}|\max(x\cdot\nu(x_1),0)^\ka\mathbf{e}(x_1)-\max(x\cdot\nu(x_2),0)^\ka\mathbf{e}(x_2)|\,dS_x\le C|x_1-x_2|^\gamma,
$$
where $\gamma=\min\{\frac{\de}{\de+2\ka},\al\}$. Combining this with the following estimate (see equation (21) in the proof of Theorem~1.4 in \cite{FotShaWei21}) \begin{align*}
\begin{multlined}|\nu(x_1)-\nu(x_2)|+|\mathbf{e}(x_1)-\mathbf{e}(x_2)|\\
\le C\int_{\p B_1}|\max(x\cdot\nu(x_1),0)^\ka\mathbf{e}(x_1)-\max(x\cdot\nu(x_2),0)^\ka\mathbf{e}(x_2)|,
\end{multlined}\end{align*}
we obtain the  H\"older continuity of $x_1\longmapsto\nu(x_1)$ and $x_1\longmapsto\mathbf{e}(x_1)$.

\medskip\noindent\emph{Step 2.} We claim that for every $\e\in(0,1)$, there exists $\rho_\e\in(0,\rho_1)$ such that for $x_1\in\Gamma^\ka(\uu)\cap\overline{B_{\rho_1}(x_0)}$ and $y\in\overline{B_{\rho_\e}(x_1)}$, \begin{align}
    &\label{eq:reg-set-cone-1}\uu(y)=0\quad\text{if } (y-x_1)\cdot\nu(x_1)<-\e|y-x_1|,\\
    &\label{eq:reg-set-cone-2}|\uu(y)|>0\quad\text{if }(y-x_1)\cdot\nu(x_1)>\e|y-x_1|.
\end{align}
Indeed, if \eqref{eq:reg-set-cone-1} does not hold, then we can take a sequence $\Gamma^\ka(\uu)\cap\overline{B_{\rho_1}(x_0)}\ni x_j\to\bar{x}$ and a sequence $y_j-x_j\to 0$ as $j\to\infty$ such that $$
|\uu(y_j)|>0\quad\text{and}\quad(y_j-x_j)\cdot\nu(x_j)<-\e|y_j-x_j|.
$$
Then we consider $\uu_j(x):=\frac{\uu(x_j+|y_j-x_j|x)}{|y_j-x_j|^\ka}$ and observe that for $z_j:=\frac{y_j-x_j}{|y_j-x_j|}\in\p B_1$, $|\uu_j(z_j)|>0$ and $z_j\cdot\nu(x_j)<-\e|z_j|$. As we have seen in the proof of Lemma~\ref{lem:c-H-unif}, the rescalings $\uu_j$ at $x_j\in\mathcal R_\uu$ converge in $C^1_{\loc}(\R^n)$ to a $\ka$-homogeneous solution $\uu_0$ of $\Delta\uu_0=f(\bar x,\uu_0)$. By applying Lemma~\ref{lem:rescaling-blowup-est} to $\uu_j$, we can see that $\uu_0(x)=\be_{\bar{x}}\max(x\cdot\nu(\bar{x}),0)^\ka\mathbf{e}(\bar{x})$. Then, for $K:=\{z\in\p B_1\,:\,z\cdot\nu(\bar{x})\le-\e/2|z|\}$, we have that $z_j\in K$ for large $j$ by \emph{Step 1}. We also consider a bigger compact set $\tilde{K}:=\{z\in\R^n\,:\,1/2\le|z|\le2, z\cdot\nu(\bar{x})\le-\e/4|z|\}$, and let $t:=\min\{\dist(K,\p\tilde{K}),r_0\}$, where $r_0=r_0(n,\al,q,M)$ is as in Lemma~\ref{imp}, so that $B_t(z_j)\subset\tilde{K}$. By applying Lemma~\ref{lem:nondeg-int}, we obtain $$
\sup_{\tilde{K}}\left(|\D\uu_j|^2+|\uu_j|^{q+1}\right)\ge c(n,t)\int_{B_t(z_j)}\left(|\D\uu_j|^2+|\uu_j|^{q+1}\right)\ge c(n,\al,q,M,\e),
$$
which gives $$
\sup_{\tilde{K}}\left(|\D\uu_0|+|\uu_0|^{q+1}\right)>0.
$$
However, this is a contradiction since $\uu_0(x)=\be_{\bar{x}}\mathbf{e}(\bar{x})\max(x\cdot\nu(\bar{x}),0)^\ka=0$ in $\tilde{K}$.

On the other hand, if \eqref{eq:reg-set-cone-2} is not true, then we take a sequence $\Gamma(\uu)\cap\overline{B_{\rho_1}(x_0)}\ni x_j\to\bar{x}$ and a sequence $y_j-x_j\to 0$ such that $|\uu(y_j)|=0$ and $(y_j-x_j)\cdot\nu(x_j)>\e|y_j-x_j|$. For $\uu_j$, $\uu_0(x)=\be_{\bar{x}}\max(x\cdot\nu(\bar{x}),0)^\ka\mathbf{e}(\bar{x})$ and $z_j$ as above, we will have that $\uu_j(z_j)=0$ and $z_j\in K':=\{z\in\p B_1\,:\,z\cdot\nu(\bar{x})\ge\e/2|z|\}$. Over a subsequence $z_j\to z_0\in K'$ and we have $\uu_0(z_0)=\lim_{j\to\infty}\uu_j(z_j)=0$. This is a contradiction since the half-space solution $\uu_0$ is nonzero in $K'$.

\medskip\noindent\emph{Step 3.} By rotations we may assume that $\nu(x_0)=\mathbf{e}_n$ and $\mathbf{e}(x_0)=\mathbf{e}_1$. Fixing $\e=\e_0$, by \emph{Step 2} and the standard arguments, we conclude that there exists a Lipschitz function $g:\R^{n-1}\to\R$ such that for some $\rho_{\e_0}>0$,\begin{align*}
    &B_{\rho_{\e_0}}(x_0)\cap\{\uu=0\}=B_{\rho_{\e_0}}(x_0)\cap\{x_n\le g(x')\},\\
    &B_{\rho_{\e_0}}(x_0)\cap\{|\uu|>0\}=B_{\rho_{\e_0}}(x_0)\cap\{x_n> g(x')\}.
\end{align*}
Now, taking $\e\to0$, we can see that $\Gamma(\uu)$ is differentiable at $x_0$ with normal $\nu(x_0)$. Recentering at any $y\in B_{\rho_{\e_0}}(x_0)\cap\Gamma(\uu)$ and using the H\"older continuity of $y\longmapsto\nu(y)$, we conclude that $g$ is $C^{1,\gamma}$. This completes the proof.
\end{proof}


\appendix
\section{Example of almost minimizers}\label{sec:ex}

\begin{example}
Let $\uu$ be a solution of the system $$
\Delta\uu+\mathbf{b}(x)\cdot\D\uu-\la_+(x)|\uu^+|^{q-1}\uu^++\la_-(x)|\uu^-|^{q-1}\uu^-=0\quad\text{in }B_1,
$$
where $\mathbf{b}\in L^p(B_1)$, $p>n$, is the velocity field. Then $\uu$ is an almost minimizer of the functional $\int\left(|\D\uu|^2+2F(x,\uu)\right)\,dx$ with a gauge function $\omega(r)=C(n,p,\|\mathbf{b}\|_{L^p(B_1)})r^{1-n/p}$.
\end{example}

\begin{proof}
Let $B_r(x_0)\Subset B_1$ and $\vv\in\uu+W^{1,2}_0(B_r(x_0))$. Without loss of generality assume that $x_0=0$. Then \begin{align*}
    &\int_{B_r}\D\uu\cdot\D(\uu-\vv)\\
    &=\int_{B_r}\mathbf{b}\cdot\D\uu\cdot(\uu-\vv)-\la_+|\uu^+|^{q-1}\uu^+\cdot(\uu-\vv)+\la_-|\uu^-|^{q-1}\uu^-\cdot(\uu-\vv)\\
    &=\int_{B_r}\mathbf{b}\cdot\D\uu\cdot(\uu-\vv)-\la_+|\uu^+|^{q+1}-\la_-|\uu^-|^{q+1}+\la_+|\uu^+|^{q-1}\uu^+\cdot\vv-\la_-|\uu^-|^{q-1}\uu^-\cdot\vv.
\end{align*}
To estimate the integral in the last line, we note that the following estimate was obtained in the proof of Example~1 in \cite{DeSJeoSha21}: $$
\int_{B_r}\mathbf{b}\cdot\D\uu\cdot(\uu-\vv)\le C(n,p,\|\textbf{b}\|_{L^p(B_1)})r^\gamma\int_{B_r}\left(|\D\uu|^2+|\D\vv|^2\right),\qquad\gamma=1-n/p.
$$
In addition, from $\uu^+\cdot\vv\le\uu^+\cdot\vv^+\le|\uu^+||\vv^+|$ and $-\uu^-\cdot\vv\le\uu^-\cdot\vv^-\le|\uu^-||\vv^-|$, we have \begin{align*}
    &\la_+|\uu^+|^{q-1}\uu^+\cdot\vv-\la_-|\uu^-|^{q-1}\uu^-\cdot\vv\\
    &\qquad\le \la_+|\uu^+|^q|\vv^+|+\la_-|\uu^-|^q|\vv^-|\\
    &\qquad\le \la_+\left(\frac{q}{1+q}|\uu^+|^{1+q}+\frac1{1+q}|\vv^+|^{1+q}\right)+\la_-\left(\frac{q}{1+q}|\uu^-|^{1+q}+\frac1{1+q}|\vv^-|^{1+q}\right),
\end{align*}
and thus \begin{align*}
    &-\la_+|\uu^+|^{q+1}-\la_-|\uu^-|^{q+1}+\la_+|\uu^+|^{q-1}\uu^+\cdot\vv-\la_-|\uu^-|^{q-1}\uu^-\cdot\vv\\
    &\qquad\le \frac{\la_+}{1+q}\left(|\vv^+|^{1+q}-|\uu^+|^{1+q}\right)+\frac{\la_-}{1+q}\left(|\vv^-|^{1+q}-|\uu^-|^{1+q}\right).
\end{align*}
Combining the above equality and inequalities yields \begin{align*}
    \int_{B_r}\D\uu\cdot\D(\uu-\vv)&\le Cr^\gamma\int_{B_r}\left(|\D\uu|^2+|\D\vv|^2\right)+\frac1{1+q}\int_{B_r}\la_+\left(|\vv^+|^{1+q}-|\uu^+|^{1+q}\right)\\
    &\qquad+\frac1{1+q}\int_{B_r}\la_-\left(|\vv^-|^{1+q}-|\uu^-|^{1+q}\right).
\end{align*}
From $$
\int_{B_r}\D\uu\cdot\D(\uu-\vv)=\int_{B_r}|\D\uu|^2-\D\uu\cdot\D\vv\ge1/2\int_{B_r}\left(|\D\uu|^2-|\D\vv|^2\right),
$$
we further have \begin{align*}
    &\int_{B_r}(1/2-Cr^\gamma)|\D\uu|^2+\frac1{1+q}\left(\la_+|\uu^+|^{1+q}+\la_-|\uu^-|^{1+q}\right)\\
    &\qquad\le \int_{B_r}(1/2+Cr^\gamma)|\D\vv|^2+\frac1{1+q}\left(\la_+|\vv^+|^{1+q}+\la_-|\vv^-|^{1+q}\right).
\end{align*}
This implies that $$
(1-Cr^\gamma)\int_{B_r}\left(|\D\uu|^2+2F(x,\uu)\right)\le (1+Cr^\gamma)\int_{B_r}\left(|\D\vv|^2+2F(x,\vv)\right),
$$
and hence we conclude that for $0<r<r_0$
$$
\int_{B_r}\left(|\D\uu|^2+2F(x,\uu)\right)\,dx\le (1+Cr^\gamma)\int_{B_r}\left(|\D\vv|^2+2F(x,\vv)\right)\,dx,
$$
with $r_0$ and $C$ depending only on $n$, $p$, $\|\mathbf{b}\|_{L^p(B_1)}$.
\end{proof}


\end{document}